\documentclass{aims}
\usepackage{amsmath}
  \usepackage{paralist}
  \usepackage{graphics}
  \usepackage{epsfig}
\usepackage{graphicx}  \usepackage{epstopdf}
 \usepackage[colorlinks=true]{hyperref}
\hypersetup{urlcolor=blue, citecolor=red}

\def\currentvolume{}

    \def\ppages{}
     \def\DOI{}

\newtheorem{theorem}{Theorem}[section]
\newtheorem{corollary}{Corollary}

\newtheorem{lemma}[theorem]{Lemma}
\newtheorem{proposition}{Proposition}

\newtheorem*{problem}{Problem}
\theoremstyle{definition}

\newtheorem{remark}{Remark}

\newcommand{\ep}{\varepsilon}

\title[Iterated quasi-reversibility method] 
      {Iterated quasi-reversibility method applied to elliptic and parabolic data completion problems}

\author[J\'er\'emi Dard\'e]{}


 \email{jeremi.darde@math.univ-toulouse.fr}

\begin{document}
\maketitle

\centerline{\scshape J\'er\'emi Dard\'e }
\medskip
{\footnotesize
 \centerline{Institut de Math\'ematiques de Toulouse ; UMR5219}
   \centerline{Universit\'e de Toulouse ; CNRS}
   \centerline{UPS IMT, F-31062 Toulouse Cedex 9, France.}
} 

\bigskip

\begin{abstract}
We study the iterated quasi-reversibility method to regularize ill-posed elliptic and parabolic problems: data
completion problems for Poisson's and heat equations. We define an abstract setting to treat both equations at once. We demonstrate the
convergence of the regularized solution to the exact one, and propose a strategy to deal with noise on the data. We present numerical experiments
for both problems: a two-dimensional corrosion detection problem and the one-dimensional heat equation with lateral data. In both cases,
the method prove to be efficient even with highly corrupted data.
\end{abstract}

\section{Introduction}
We consider data completion problems for
elliptic and parabolic operators. We start with elliptic operators: we consider a bounded domain $\Omega \subset \mathbb{R}^d$, $d \geq 2$, 
with Lipschitz boundary (see \cite{Grisv}). Let $\nu \in L^\infty(\partial \Omega, \mathbb{R}^d)$ be the exterior unit normal of $\partial \Omega$, 
and $\Gamma,\ \Gamma_c \subset \partial \Omega$,
such that $\partial \Omega = \overline{\Gamma \cup \Gamma_c}$ and $meas(\Gamma), \ meas(\Gamma_c) > 0$. Let $\sigma : \Omega \mapsto \mathbb{R}^{d\times d}$ be
a real matrix valued function such that $\sigma \in W^{1,\infty}(\Omega)^{d\times d}$ and 
$$
\sigma = \sigma^T,\quad c\, \vert \xi \vert^2 \leq  \sigma \xi \cdot \xi,\ \forall \xi \in \mathbb{R}^d, \ \textit{a.e. in } \Omega. 
$$ 
The data completion problem is:
\begin{problem}
For $f$, $g_D$ and $g_N$ in $L^2(\Omega) \times L^2(\Gamma) \times L^2(\Gamma)$, find $u \in H^1(\Omega)$ such that
$$
\left\lbrace
\begin{array}{ccl}
- \nabla \cdot \sigma \nabla u &=& f \text{ in } \Omega \\
 u &= &g_D \text{ on } \Gamma \\
 \sigma \nabla u \cdot \nu &= &g_N \text{ on } \Gamma.
\end{array}
\right.
$$
\end{problem}

This problem is well-known to be ill-posed (see \cite{AlesRond,BenBelg} and the references therein): it does not necessarily admit a solution for any data
$(f,g_D,g_N)$, and if a solution
exists, it does not depend continuously on the data. 
On the other hand, if the problem admits a solution
$u_s$, this solution is necessarily unique (see e.g. \cite{AlesRond,LeRouss}).

Such a problem is encountered  in many practical applications, among others in \textit{plasma physic} \cite{Faug,Blum}, or \textit{corrosion detection
problems}
\cite{Aless, Sinc, Bara, Fasi, Chaa}.
We will be particularly interested in the corrosion
detection problem: in this problem, $u$ is the electrical potential inside a conductive object $\Omega$, $\sigma$ is the conductivity
of the object, $g_N$ represent a current imposed on $\Gamma$, accessible part of the boundary of $\Omega$, and $g_D$ is the corresponding 
potential measured on $\Gamma$. The aim is to determine if some portion of the inaccessible part of the boundary $\Gamma_c$ is corroded. 

Mathematically, it exists a non-negative function $ \mu$  define on $\Gamma_{c}$ such that
$$\displaystyle \sigma \nabla u\cdot \nu + \mu \, u =0 \text{ on } \Gamma_c$$
and the objective is to reconstruct $\mu$: $\mu = 0$ on the healthy part of $\Gamma_c$, and $\mu >0$ on the corroded part. In section \ref{sect_num_ell}, we  test our method on this problem.

The data completion problem is known to be severely, even \textit{exponentially} ill-posed \cite{BenBelg}. Therefore
 ones needs to use regularization methods to try to reconstruct $u$. Several methods have been proposed to stabilize the problem:
 see, e.g., \cite{Andr,Cime,BenBelgThang,Aza,Haddar,Burman} and the references therein.

We are also interesting in the data completion problem for the heat equation, which is quite similar to the elliptic one,
except that this time $u$ solves a parabolic equation. Such inverse problem appears naturally in \textit{thermal imaging} \cite{BryCau} and
\text{inverse obstacle problems} \cite{HarbTau,IkeKawa}. For $T>0$, we define $\mathcal{Q} := (0,T) \times \Omega$. Let
$f$ be in $ L^2(\mathcal{Q})$, $g_D$ and $g_N$ in $L^2(0,T; L^2(\Gamma))$. The data completion problem is then

\begin{problem}
find $u \in H^{1,1}(\mathcal{Q}) := L^2(0,T; H^1(\Omega)) \cap H^1(0,T; L^2(\Omega))$ such that
$$
\left\lbrace
\begin{array}{rcl}
\partial_t u - \Delta u &=& f \text{ in } \mathcal{Q} \\
u &=& g_D \text{ on } (0,T) \times \Gamma \\
\nabla u \cdot \nu &= &g_N \text{ on } (0,T) \times \Gamma 
\end{array}
\right.
$$
\end{problem}

This parabolic data completion problem is also severely ill-posed (see e.g. \cite{Puzy}).  Note that it is not mandatory to impose an initial condition
$u(0,.)$ on $\Omega$ to obtain the uniqueness of the solution (if such a solution exists). Again, regularization methods are needed to obtain a stable reconstruction of $u$ from the data $f,\, g_D$ and $g_N$.

The \emph{quasi-reversibility method} is such a regularization method, introduced in the pioneering work of Latt\`es and Lions \cite{Latt} to regularize
elliptic, parabolic (and even hyperbolic) data completion problems.
 The mean idea of the method is to approach the ill-posed data completion problem by a family of  well-posed variational
 problems of higher order (typically fourth order problems) depending on a (small) parameter $\ep$. The solution of the regularized problem converges to the solution of
the data completion problem, when the parameter $\ep$ goes to zero. The quasi-reversibility method presents interesting
features: first of all the variational problems appearing in the method are naturally discretized using \textit{finite element methods},  thus the method can be used in complicated geometries, an interesting property when the method is used in an iterative algorithm with changing domain.  Furthermore, the method is independent of the dimension.
Since its introduction, the quasi-reversibility method has been successfully used to reconstruct the solution of elliptic 
\cite{KlibSant,Clason,Bourgeois2,CaoKlib,BourDar}
and parabolic \cite{Ames,Clark} ill-posed problems, and as a keystone in the resolution of inverse obstacle problems in the \textit{exterior approach}
\cite{BourDar2,Dar,BourDar3}.

 In the present paper, we are interested in a natural extension of the quasi-reversibility method, the 
\textit{iterated quasi-reversibility method}: it consists in solving iteratively quasi-reversibility problems, the solution of each one depending
on the solution of the previous one.  We therefore obtain a sequence of quasi-reversibility solutions, which converges to the exact solution
of the data completion problem if exact data are provided, for any choice of the regularization parameter $\ep$. This has interesting consequences
from a numerical point of view: first of all, one can now choose a large value for the parameter of regularization
$\ep$, leading to an improvement in the conditioning of the finite-element problems, without lowering the quality of the reconstruction.  This is not the case for the standard quasi-reversibility method,
for which it is mandatory to use small $\ep$ to obtain a good reconstruction. 
Furthermore, in presence of noisy data, we present a method to choose
when to stop the iterations according to the amplitude of noise on the data, based on the \textit{Morozov discrepancy principle}, which ensure both
stability and convergence of the method.
The main drawback of this extension of the quasi-reversibility method, comparatively to the standard quasi-reversibility, is that several problems have to be solved to obtain a good reconstruction. However, as it is the same variational problem that appears in each iteration of the method, one can precompute a factorization of the finite-element
matrix. Hence, the cost of the method is not significantly higher.

The paper is organized as follows. In section 2, we introduce an abstract setting to treat both data completion problems
we are interested in at once. 
In section 3, we present the standard quasi-reversibility regularization in this abstract setting, and prove some results we need to study
the iterated quasi-reversibility method. In section
4, we focus on the  iterated quasi-reversibility method, both in the case of exact data and noisy data. In section 5, we show
that the abstract setting apply to both elliptic and parabolic data completion problems.
In section 6, numerical results are presented, demonstrating the feasibility and efficiency of the method for both problems.

\section{An abstract setting for data completion problems}

In this section, we set up an abstract setting corresponding to both data completion problems we are interested in.

Let $\mathcal{X}$, $\mathcal{Y}$ be two Hilbert spaces endowed with respective scalar products $(.,.)_\mathcal{X}$ and $(.,.)_\mathcal{Y}$, and corresponding
norms denoted $\Vert . \Vert_\mathcal{X}$ and $\Vert . \Vert_\mathcal{Y}$. 

Let $y\in \mathcal{Y}$. Both of our data completion problems can be written in the following way: find $x \in \mathcal{X}$
such that $A x= y$, with   $A : \mathcal{X} \mapsto \mathcal{Y}$ a continuous linear operator with 
following properties:
\begin{itemize}
\item $A$ is one-to-one
\item $A$ is not onto
\item $\overline{\mathsf{Im}(A)}^{\mathcal{Y}} = \mathcal{Y}$.
\end{itemize}
In this setting, $y$ plays the role of the data, and $x$ the solution of our data completion problem.
The problem is obviously ill-posed: indeed, as $A$ is not onto, there exist $y$ in $\mathcal{Y}$ for which the problem admits no solution.
We define $\mathcal{Y}_{adm} := \mathsf{Im}(A)$ the set of \textit{admissible} data, and $\mathcal{Y}_{nadm} = \mathcal{Y} \setminus \mathcal{Y}_{adm}$ the set
of \textit{non-admissible} ones. By definition, $\mathcal{Y}_{adm}$ is dense in $\mathcal{Y}$. Actually, this is also true for $\mathcal{Y}_{nadm}$

\begin{proposition}
The set $\mathcal{Y}_{nadm}$ is dense in $\mathcal{Y}$.
\end{proposition}

\begin{proof}
This is quite simple: suppose it exists $\bar{y} \in \mathcal{Y}_{adm}$ and $\delta > 0$ such that $\Vert y - \bar{y} \Vert_\mathcal{Y} \leq \delta \Rightarrow y \in \mathcal{Y}_{adm}.$
It exists $\bar{x} \in \mathcal{X}$ s.t. $A \bar{x} = \bar{y}$.

Let $y$ be any element of $\mathcal{Y}$, $y \neq \bar{y}$. We define $\tilde{y} = \displaystyle \frac{y - \bar{y}}{\Vert y - \bar{y}\Vert_\mathcal{Y}} \frac{\delta}{2} + \bar{y}$.
Obviously, $\Vert \tilde{y} - \bar{y}\Vert_\mathcal{Y} \leq \delta$. Therefore, $\tilde{y} \in \mathcal{Y}_{adm}$, and it exists
$\tilde{x} \in \mathcal{X}$ such that $A \tilde{x} = \tilde{y}$. A simple computation shows then that
$$
A \Big( \frac{2 \Vert y - \bar{y} \Vert_\mathcal{Y}}{\delta} (\tilde{x} - \bar{x}) + \bar{x} \Big) = y.
$$
Hence $Im(A) = \mathcal{Y}$, contradicting the assumptions on $A$. Therefore,
for any $y \in \mathcal{Y}_{adm}$, for any $\delta>0$, there exists $y^\delta \in \mathcal{Y}_{nadm}$ such that 
$\Vert y - y^\delta \Vert_\mathcal{Y} \leq \delta$, which ends the proof, as $\mathcal{Y} = \mathcal{Y}_{adm} \cup \mathcal{Y}_{nadm}$.
\end{proof}

In other word, for any admissible data $y$  exists a non-admissible one $\tilde{y}$ arbitrary close to
$y$. In particular, this leads to the high instability of the problem with respect to noise:

\begin{proposition} \label{prop_unstability}
For any $y \in \mathcal{Y}$, the exists  a sequence $x_n \in \mathcal{X}$ such that 
$$
\Vert x_n \Vert_\mathcal{X} \xrightarrow[]{n \rightarrow \infty}+ \infty \text{ and } Ax_n \xrightarrow[\mathcal{Y}]{n \rightarrow \infty} y.
$$
\end{proposition}

\begin{proof}
We start with $y\in \mathcal{Y}_{nadm}$. As $\textsf{Im}(A)$ is dense in $\mathcal{Y}$, it exists a sequence $x_n \in \mathcal{X} $ in such that $A x_n 
\xrightarrow[\mathcal{Y}]{n \rightarrow \infty} y$. This sequence cannot have any bounded subsequence: indeed, if 
such a subsequence would exist, there would be another subsequence, denoted $x_m$ here, such that
$x_m$ weakly converges to an element $x$ in $\mathcal{X}$. The operator $A$ being linear and \textit{strongly}
continuous, it is \textit{weakly} continuous \cite{Brez}, hence $A x_m$ weakly converges to $A x$. But by definition
$A x_m$ strongly converges to $y$. By uniqueness of the limit, we have $Ax = y$, and $y \in \mathcal{Y}_{adm}$,
in contradiction with the initial assumption. Therefore, we have $\Vert x_n \Vert_\mathcal{X} \xrightarrow[]{n\rightarrow \infty} + \infty$.

Now, consider $y \in \mathcal{Y}_{adm}$. The previous proposition implies the existence of a sequence $y_m \in \mathcal{Y}_{nadm}$
such that $y_m \xrightarrow[\mathcal{Y}]{m \rightarrow \infty}y$. For a fixed $m$, we now know the existence of a sequence 
$x_{m,n}\in \mathcal{X}$ such that $A x_{m,n} \xrightarrow[\mathcal{Y}]{n \rightarrow \infty} y_m$ and $\Vert x_{m,n} \Vert_\mathcal{X} \xrightarrow[]{n \rightarrow \infty} + \infty$.
In particular, for any $m \in \mathbb{N}$, there exists $n(m) \in \mathbb{N}$ such that $\tilde{x}_m := x_{m,n(m)}$ verifies
at the same time
$$
\Vert \tilde{x}_m \Vert_\mathcal{X} \geq m \text{ and } \Vert A \tilde{x}_m - y_m \Vert_\mathcal{Y} \leq \frac{1}{m}.
$$
It is then not difficult to verify that the sequence $\tilde{x}_m$ verifies the researched properties.
\end{proof}

\begin{remark} \label{rk_unstability}
Actually, if $y$ is not an admissible data, it is shown in the proof that any sequence $(x_n)_{n \in \mathbb{N}} \in \mathcal{X}^\mathbb{N}$ such that
$A x_n \xrightarrow[\mathcal{Y}]{n \rightarrow \infty} y$ verifies $\displaystyle \lim_{n \rightarrow \infty} \Vert x_n \Vert_\mathcal{X} = + \infty $.
\end{remark}

This proposition has for important consequences the fact that for any \textit{admissible} data $y$, with 
corresponding solution $x$, one can find an \textit{admissible} data $\tilde{y}$, with corresponding solution
$\tilde{x}$, such that $\tilde{y}$ is \textit{arbitrarily} close to $y$ and $\tilde{x}$ is \textit{arbitrarily} far
from $x$.

We retrieve here the well-known fact that the problem of noisy data is crucial in  data completion problems.
Clearly, it is not sufficient to build a method that (approximately) reconstruct the solution of the data completion problem
for any admissible data, it is also mandatory to propose a strategy for noisy data, as in practice data are always corrupted by some noise
due to inaccurate measurements.

\section{Standard quasi-reversibility method}

We define $b$ 
a symmetric bilinear non-negative form
on $\mathcal{X}$, and denote by $\Vert . \Vert_b$ the induced seminorm on $\mathcal{X}$. We suppose
that it exists two strictly positive constants $c,\ C$ such that
$$
c^2 \Vert x \Vert_\mathcal{X}^2 \leq \Vert A x \Vert_\mathcal{Y}^2 + \Vert x\Vert_b^2 \leq C^2 \Vert x \Vert_\mathcal{X}^2.
$$
Therefore,  the symmetric bilinear form $( . , .)_{A,b}$, define by
$$
\forall (x,\tilde{x}) \in \mathcal{X}, (x,\tilde{x})_{A,b} = (Ax,A\tilde{x})_Y + b(x,\tilde{x}),
$$
is a scalar product on $\mathcal{X}$, and
$\mathcal{X}$ endowed with this scalar product is a Hilbert space. We denote $\Vert .\Vert_{A,b}$ the corresponding norm, which
is equivalent to the $\Vert . \Vert_\mathcal{X}$ norm.

Obviously, there exists such a form $b$: it suffices to take the whole scalar product in $\mathcal{X}$, $b(.,.) = (.,.)_\mathcal{X}$.

Adapting the initial idea of Jacques-Louis Lions and Robert Latt\`es \cite{Latt}, the quasi-reversibility method applied to  the abstract data
completion problem defined above  relies on the resolution of the following regularized problem
\begin{problem}
for $y \in \mathcal{Y}$ and $\ep>0$, find $x_\ep \in \mathcal{X}$ such that 
$$
(A x_\ep, A x)_\mathcal{Y} + \ep\, b(x_\ep, x) = (y,A x)_\mathcal{Y},\quad \forall x \in \mathcal{X}.
$$
\end{problem}
The quasi-reversibility equation is the Euler-Lagrange equation corresponding to the minimization over $\mathcal{X}$
of the energy 
$
\Vert A x - f \Vert_\mathcal{Y}^2 + \ep \Vert x \Vert_b^2
$. In other words, it is a \textit{Tykhonov regularization} of the data completion problem, $\ep>0$ being the parameter of regularization and 
$\Vert . \Vert_b$ the penalization (semi)norm. 
Since its introduction in 1963 by A.N. Tykhonov \cite{Tykh}, this regularization has been widely studied and used to solve inverse problems
(for a complete study on the topic, see \cite{Engl}  and the references therein). 
There are various methods to study such regularization method: e.g. \textit{singular value decomposition} if $A$ is compact (which is not the case
in our data completion problems, see section \ref{sect_QRconcrete}) or \textit{spectral theory}. 
Here we propose another approach to study the method, based on the variational formulation of the quasi-reversibility method,
and on the differentiability of the approximated solution with respect to the parameter of regularization, the later
being useful in the study of the iterated quasi-reversibility method.

First of all, let us verify that the quasi-reversibility problem is well-posed.

\begin{proposition} \label{prop_QR11}
For any $y \in \mathcal{Y}$ and $\ep>0$, the quasi-reversibility problems admits a unique solution $x_\ep$, with the following estimates:
$$
\Vert A x_\ep \Vert_\mathcal{Y} \leq \Vert f \Vert_\mathcal{Y}, \quad \Vert A x_\ep - y \Vert_\mathcal{Y} \leq \Vert y \Vert_\mathcal{Y},\quad \Vert x_\ep \Vert_{A, b} \leq \frac{1}{\min (1,\sqrt{\ep})}
\Vert y \Vert_\mathcal{Y}.
$$
\end{proposition}

\begin{proof}
 let us define the bilinear form
$$
a_\ep(x,\tilde{x}) := (Ax,A\tilde{x})_\mathcal{Y} + \ep\ b(x,\tilde{x}),\quad \forall x,\ \tilde{x} \in \mathcal{X}.
$$
It is obviously continuous. Furthermore, for all $x\in X$, we have
$$
a_\ep(x,\tilde{x}) \geq \min(1, \ep) \|x\|_{A,b}^2,
$$
and therefore it is coercive. Finally, as $|(y,A x)_\mathcal{Y}| \leq \|A\|\, \|y\|_\mathcal{Y}\, \|x\|_\mathcal{X} \leq  \|A\|\, \|y\|_\mathcal{Y}\, \|x\|_{A,b}$, we obtain the existence and uniqueness of $x_\ep$ by Lax-Milgram theorem. By definition, we have
$$
\|A x_\ep \|_\mathcal{Y}^2\leq \|A x_\ep \|_\mathcal{Y}^2 + \ep \|x_\ep\|_b^2 = (Ax_\ep,y)_\mathcal{Y} \leq \|A x_\ep\|_\mathcal{Y} \|y\|_\mathcal{Y} \Rightarrow \|A x_\ep \|_\mathcal{Y} \leq \|y\|_\mathcal{Y}.
$$
Furthermore,
$$
(Ax_\ep - y,Ax_\ep)_\mathcal{Y} = -\ep \|x_\ep\|_b^2 \leq 0 \Rightarrow \|Ax_\ep- y\|_\mathcal{Y}^2 \leq -(y,Ax_\ep-y)_\mathcal{Y} \leq \|y \|_Y
\|Ax_\ep - y\|_\mathcal{Y},
$$
implying $\|Ax_\ep - y\|_\mathcal{Y} \leq \|y\|_\mathcal{Y}$.
Finally, we have
$$
\|A x_\ep \|_\mathcal{Y}^2 + \ep \|x_\ep\|_b^2 = (Ax_\ep,y)_\mathcal{Y} \leq \|A x_\ep\|_\mathcal{Y} \|y\|_\mathcal{Y} \leq \sqrt{\|A x_\ep \|_\mathcal{Y}^2 + \ep\|x_\ep\|_b^2}\, \|y\|_\mathcal{Y}
$$
leading to 
$
\displaystyle \min (1,\sqrt{\ep}) \|x_\ep\|_{A,b} \leq \sqrt{\|A x_\ep \|_\mathcal{Y}^2 + \ep \|x_\ep\|_b^2} \leq \|y\|_Y.
$
\end{proof}

\begin{remark}
In particular, we always have $x_\ep \xrightarrow[\ep \rightarrow \infty]{\mathcal{X}} 0$.
\end{remark}

Suppose there exists $x \in \mathcal{X}$ such that $Ax = y$ (\textit{i.e.} $y \in \mathcal{Y}_{adm}$). It is easily seen that $x$ is \textit{never}
the solution of the quasi-reversibility problem, except in the special case $y = 0$ (which is always in $\mathcal{Y}_{adm}$) for which $x= 0 = x_\ep$. In other words, there is no $\ep >0$ such that the quasi-reversibility method reconstructs exactly the exact solution of the data completion problem.
As seen in the following corollary, the solution of the quasi-reversibility problem is also never $0$, except again in the special case $y = 0$.

\begin{corollary} \label{coro_1}
The three following properties are equivalent:
\begin{itemize}
\item[(i)] $y \neq 0$
\item[(ii)] $\exists\ \ep > 0$ s.t. $x_\ep \neq 0$
\item[(iii)] $\forall \ep > 0$, $x_\ep \neq 0$.
\end{itemize}
\end{corollary}

\begin{proof}
obviously, (iii) implies (ii). Furthermore, as $\min( 1, \sqrt{\ep})  \|x_\ep\|_{A,b} \leq\|y\|_Y$, (ii) implies (i). 

Suppose it exists $\ep>0$ such that $x_\ep = 0$. For that particular $\ep$ and for any $x \in X$, we would have
$
(y,Ax)_\mathcal{Y}= (Ax_\ep,Ax)_\mathcal{Y} + \ep (x_\ep,x)_b = 0.
$
As $\overline{Im(A)}^\mathcal{Y} = \mathcal{Y}$, this directly implies $y=0$, so (i) implies (iii).
\end{proof}

\begin{proposition} \label{prop_conv_Axe}
Let $y \in \mathcal{Y}$, and $x_\ep$ the solution of the corresponding quasi-reversibility problem. Then $A x_\ep$ strongly converges to $y$ (even if
$y$ is not an admissible data).
\end{proposition}

\begin{proof}
As $\displaystyle \min( 1, \sqrt{\ep} )\Vert x_\ep \Vert_{A, b} \leq \Vert y \Vert_\mathcal{Y}$, we have, for any $x \in \mathcal{X}$,
$$
\displaystyle \ep \vert\,b(x_\ep,x) \vert \leq \frac{\ep}{\min (1, \sqrt{\ep} )} \Vert y \Vert_\mathcal{Y} \Vert  x\Vert_b \xrightarrow[]{\ep \rightarrow 0} 0.
$$
Let $(\ep_m)_{m \in \mathbb{N}}$ be a decreasing sequence of strictly positive real numbers such that $\lim_{m\rightarrow \infty} \ep_m = 0$,
and note $x_m := x_{\ep_m}$. As $\Vert A x_m \Vert_\mathcal{Y} \leq \Vert y \Vert_\mathcal{Y}$, it exists a subsequence (still denoted $x_m$) such that 
$Ax_m$ weakly converges to $\tilde{y}\in \mathcal{Y}$. But, for all $x\in \mathcal{X}$, we have
$$
(y - \tilde{y}, Ax)_\mathcal{Y} \xleftarrow[]{m \rightarrow \infty} (y - Ax_m, Ax)_\mathcal{Y} = \ep_m b(x_m,x) \xrightarrow[]{m \rightarrow \infty} 0,
$$
that is $y = \tilde{y}$ as $\overline{\mathsf{Im}(A)}^\mathcal{Y} = \mathcal{Y}$, and $A x_m$ weakly converges to $y$. 
As $\Vert A x_m \Vert_\mathcal{Y} \leq \Vert y \Vert_\mathcal{Y}$ (proposition \ref{prop_QR11}),  $A x_m$ strongly converges to $y$.
It is then not difficult to see that $A x_\ep$ strongly converges to $y$ as $\ep$ goes to zero.
\end{proof}

We can now state the  main theorem regarding the standard quasi-reversibility method:

\begin{theorem}
Suppose $y \in \mathcal{Y}_{adm}$, and let $x_s$ be the (necessarily unique) solution of the abstract data completion problem. Then $x_\ep$ 
converges to $x_s$ as $\ep$ goes to zero, and we have the estimates $\Vert A x_\ep - y \Vert_\mathcal{Y} \leq \sqrt{\ep} \Vert x_s \Vert_b$,
 $\Vert x_\ep \Vert_b \leq \Vert x_s \Vert_b$ and $\Vert x_\ep - x_s \Vert_b \leq \Vert x_s \Vert_b$.

Suppose $y \in \mathcal{Y}_{nadm}$. Then $\displaystyle \lim_{\ep \rightarrow 0} \Vert x_\ep \Vert_b = +\infty$. 

The theorem remains  valid when the $\Vert . \Vert_b$ seminorm is replaced with the $\Vert.\Vert_{A,b}$ norm.
\end{theorem}

\begin{proof}
Suppose first that $y \in \mathcal{Y}_{nadm}$. Then, as $x_\ep$ is a sequence in $\mathcal{X}$ such that $A x_\ep$ converges to $y$ (proposition \ref{prop_conv_Axe}), proposition 
\ref{prop_unstability} and remark \ref{rk_unstability} imply $\lim_{\ep \rightarrow 0} \Vert x_\ep \Vert_\mathcal{X} = +\infty$.
As $\Vert x_\ep \Vert_{A,b}^2 = \Vert A x_\ep \Vert_{\mathcal{Y}}^2  + \Vert x_\ep \Vert_{b}^2 \geq c^2 \Vert x_\ep \Vert_\mathcal{X}$, we have
$ \lim_{\ep \rightarrow 0} \Vert x_\ep \Vert_{b} = +\infty$. 

Now, suppose it exists $x_s$ such that $A x_s = y$. Then, choosing $x = x_\ep - x_s$ as test function in the quasi-reversibility problem, we obtain
\begin{equation} \label{ref_eq1_proof_thm1}
\Vert A x_\ep - y \Vert_\mathcal{Y}^2 + \ep b(x_\ep, x_\ep - x_s) = 0,
\end{equation}
which in turn implies $b(x_\ep, x_\ep - x_s) \leq 0 \Rightarrow \Vert x_\ep \Vert_b \leq \Vert x_s \Vert_b
\Rightarrow \Vert x_\ep \Vert_{A,b} \leq \Vert x_s \Vert_{A,b}$. Therefore, $x_\ep$ is a bounded
sequence in $\mathcal{X}$, and up to a subsequence it weakly converges to $\tilde{x}$. 
As $A$ is a linear continuous operator, and hence is weakly continuous, proposition \ref{prop_conv_Axe} implies $A \tilde{x} = y$, which 
implies $\tilde{x} = x_s$ as $A$ is one-to-one. The uniqueness of the limit implies that the whole sequence  weakly converges to $x_s$. Finally
as $\Vert x_\ep \Vert_{A,p} \leq \Vert x_s \Vert_{A,p}$, the sequence strongly converges to $x_s$.

Subtracting $\ep b(x_s, x_\ep - x_s)$ to equation \ref{ref_eq1_proof_thm1}, we obtain
$$
\Vert A x_\ep - y \Vert_\mathcal{Y}^2 + \ep \Vert x_s - x_\ep \Vert_b^2 = - \ep b(x_s, x_\ep - x_s) \Rightarrow \Vert x_s - x_\ep \Vert_b^2 \leq
\vert b(x_s, x_\ep - x_s) \vert
$$ 
and by Cauchy-Schwarz inequality, $\Vert x_\ep - x_s \Vert_b \leq \Vert x_s \Vert_b$.

Finally, equation \ref{ref_eq1_proof_thm1} implies
$$
\Vert Ax_\ep - y \Vert_\mathcal{Y}^2 \leq \ep \Vert x_\ep \Vert_b\, \Vert x_\ep - x_s \Vert_b \leq \ep \Vert x_s \Vert^2_b
$$
which ends the proof.
\end{proof}

Next, we focus on the differentiability of the solution of  the quasi-reversibility
method with respect to $\varepsilon$, a result that will be useful in the study of the iterated quasi-reversibility method.

\subsection{Differentiability of the quasi-reversibility solution with respect to $\ep$}

It turns out that $x_\ep$, solution of the quasi-reversibility problem, depends smoothly on the parameter of regularization $\ep$.
Indeed, let us define the map $F : \ep >0 \mapsto x_\ep$. 

\begin{proposition}
The map $F$ is continuous.
\end{proposition}

\begin{proof}
We choose $\ep >0$ and $h$ such that $\ep - \vert h \vert > 0$. For any $x \in \mathcal{X}$, we have
\begin{align*}
(A x_{\ep+h},Ax)_\mathcal{Y} &+ (\ep+h)\, b(x_{\ep+h},x)  = (y,Ax)_\mathcal{Y}\\
(Ax_\ep,Ax)_\mathcal{Y} &+ \ep\ b(x_\ep,x)  = (y,Ax)_\mathcal{Y}.
\end{align*}
Subtracting the two equations, and choosing $x = \tilde{x}_{\ep,h} := x_{\ep+h}- x_\ep$, lead to
$$
\|A \tilde{x}_{\ep,h} \|_\mathcal{Y}^2 + \ep\, \|\tilde{x}_{\ep,h} \|_b^2 = -h\, b (x_{\ep+h},\tilde{x}_{\ep,h}).
$$
In conclusion, we have
$$
\min (1,\ep)\, \|\tilde{x}_{\ep,h} \|_{A,b}^2 \leq h \|x_{\ep+h}\|_b \|\tilde{x}_{\ep,h} \|_b \leq h \min(1, (\ep + h)^{-1/2})
\|y \|_\mathcal{Y} \|\tilde{x}_{\ep,h} \|_{A,b} 
$$
which ends the proof. 
\end{proof}

\begin{remark}
If the data completion problem admits a solution $x_s$, then $F$ extends continuously to $\mathbb{R}^+$ by defining
$F(0) = x_s$.
\end{remark}

\begin{proposition}
We have $F \in C^1(\mathbb{R}^+_*;\mathcal{X})$. For all $\ep>0$, $F'(\ep) = x_\ep^{(1)}$, unique element of $\mathcal{X}$ verifying 
\begin{equation} \label{eq_first_der}
(A x_\ep^{(1)}, A x)_\mathcal{Y} + \ep b (x_\ep^{(1)},x) = - b(x_\ep,x), \quad \forall x \in \mathcal{X}.
\end{equation}
Furthermore, $\Vert x_\ep^{(1)}\Vert_{A,b} \leq \min(1,\ep^{-\frac{3}{2}}) \Vert y \Vert_\mathcal{Y}$.
\end{proposition}

\begin{proof}
By Lax-Milgram theorem, there exists a unique $x_\ep^{(1)} \in \mathcal{X}$ verifying
\ref{eq_first_der}, and it clearly verifies 
$$\min (1,\ep) \|x_\ep^{(1)}\|_{A,b}^2 \leq \|x_\ep \|_b \Vert x_\ep^{(1)} \Vert_{A,b} \leq \min(1,\ep^{-\frac{1}{2}})\, \| y\|_\mathcal{Y} \Vert x_\ep^{(1)} \Vert_{A,b}.$$
It is a continuous function of $\ep$: indeed, for $\ep>0$ and $h\in\mathbb{R}$ s.t. $\ep - |h| >0$, we have, for all
$x \in \mathcal{X}$,
\begin{align*}
(A x_{\ep+h}^{(1)},Ax)_\mathcal{Y}   + (\ep+h)\, & b(x_{\ep+h}^{(1)},x) = -b (x_{\ep+h},x)\\
(A x_\ep^{(1)},Ax)_\mathcal{Y}  + \ep\ & b(x_\ep^{(1)},x)  = -b(x_\ep,x).
\end{align*}
Choosing $x = \tilde{x}_{\ep,h}^{(1)} := x_{\ep+h}^{(1)} - x_\ep^{(1)} \in\mathcal{X}$ and subtracting the two equations lead to
$$
\|A \tilde{x}_{\ep,h}^{(1)}\|_\mathcal{Y}^2 + \ep \| \tilde{x}_{\ep,h}^{(1)}\|_b^2 = -h b(x_{\ep+h}^{(1)},\tilde{x}_{\ep,h}^{(1)}) - b (\tilde{x}_{\ep,h},
 \tilde{x}_{\ep,h}^{(1)}). 
$$
Therefore, 
$$
\|\tilde{x}_{\ep,h}^{(1)} \|_{A,b} \leq h \Big( \min(1,(\ep+h)^{-3/2}) + \min(1, (\ep + h)^{-1/2}) \min(1,\ep^{-1} )  \Big)  \|y \|_\mathcal{Y}
$$
implying the continuity of the map $\mathbb{R}^+_* \ni \ep \mapsto x_{\ep}^{(1)} \in \mathcal{X}$. Remains to be proved that $F'(\ep) = x_{\ep}^{(1)}$. For $\ep> 0$ and $h \in \mathbb{R}$ such that $\ep - |h| >0$, we have
\begin{align*}
(A x_{\ep+h},Ax)_\mathcal{Y} &+ (\ep+h)\,b (x_{\ep+h},x)  = (y,Ax)_\mathcal{Y}\\
-(Ax_\ep,Ax)_\mathcal{Y} &- \ep\ b(x_\ep,x)  = -(y,Ax)_\mathcal{Y} \\
-h\, (A x_\ep^{(1)},Ax)_\mathcal{Y} & - h\, \ep\ b(x_\ep^{(1)},x)  = h\, b(x_\ep,x).
\end{align*}
Choosing $x = \hat{x}_{\ep,h} := x_{\ep+h} - x_\ep - h x_\ep^{(1)}$ and adding the three above relations lead to
$$
\|A \hat{x}_{\ep,h} \|_\mathcal{Y}^2 + \ep \|\hat{x}_{\ep,h} \|_b^2 = - hb (\tilde{x}_{\ep,h},\hat{x}_{\ep,h})
\Rightarrow \|\hat{x}_{\ep,h}\|_{A,b} \leq  h^2 \ C(h,\ep)\|y \|_\mathcal{Y}, \text{ with }  C(h,\ep) > c > 0.
$$
The result follows.
\end{proof}

A simple induction leads then to the following theorem:

\begin{theorem}
$F \in C^\infty(\mathbb{R}^+_*;\mathcal{X})$. For $\ep>0$, for all $m \in \mathbb{N}$, 
$$
\frac{d^m F}{d\ep^m}(\ep) := x^{(m)}_\ep
$$
with $x_\ep^{(m)}$ defined recursively by
$$
\left\lbrace
\begin{array}{l}
x_\ep^{(0)} := x_\ep, \\
\forall m\in \mathbb{N},\ x_\ep^{(m+1)} \text{ is the only element of } \mathcal{X} \text{ verifying }\\
 (A x_{\ep}^{(m+1)},A x)_\mathcal{Y} + \ep\, b(x_\ep^{(m+1)},x)  = - (m+1)\, b(x_\ep^{(m)},x),\ \forall x \in \mathcal{X}.
\end{array}
\right.
$$
In particular, $x_\ep^{(m)}$ verifies the following estimate:
$$
\|x_\ep^{(m)} \|_{A,b} \leq \frac{m\,!}{\min(1,\ep^{m+1/2})} \|y \|_\mathcal{Y}.
$$
\end{theorem}

If the data completion problem admits a solution $x_s$,  it is not difficult to prove that 
$$\Vert x_\ep^{(m)} \Vert_{A,b} \leq \min(1,\ep^{m})^{-1} \, m!\, \Vert x_s\Vert_{A,b}.$$
Finally, we have the following generalization of corollary \ref{coro_1}:

\begin{corollary} \label{coro_2}
the three following properties are equivalent:
\begin{itemize}
\item[(i)] $y \neq 0$
\item[(ii)] $\exists \ep >0$, $\exists m \in \mathbb{N}$ s.t. $x_\ep^{(m)} \neq 0$
\item[(iii)] $\forall \ep >0$, $\forall m \in \mathbb{N}$, $x_\ep^{(m)} \neq 0$.
\end{itemize}
\end{corollary}

\begin{proof}
Clearly (iii) implies (ii). Furthermore, as $
\displaystyle  \min(1, \ep^{m+\frac{1}{2}})\, \|x_\ep^{(m)} \|_{A,b} \leq m\,!\,  \|y \|_\mathcal{Y}
$, (ii) implies~(i).

Suppose it exists $\ep >0$ and $m \in \mathbb{N}$ such that $x_\ep^{(m)}= 0$. If $m=0$, then corollary \ref{coro_1} implies $y=0$. If $m >0$, as $(Ax_{\ep}^{(m)},Ax_{\ep}^{(m-1)})_\mathcal{Y} + \ep\, b(x_\ep^{(m)},x_\ep^{(m-1)})
 = - m \|x_\ep^{(m-1)}\|_b^2$, we obtain $x_\ep^{(m-1)}=0$, and by induction $x_\ep = 0$, implying again $y=0$. Therefore (i) implies (iii). 
\end{proof}

\subsection{Monotonic convergence of the quasi-reversibility method}

In this section, $y \neq 0$.
Using the results on the derivatives of $x_\ep$ with respect to $\ep$, it is easy to prove that if the data completion problem
admits a solution $x_s$, then $x_\ep$ converges monotonically to $x_s$ when $\ep$ goes to zero.
This is of course not the only method to obtain such results (see for example \cite{Engl}, where spectral theory is used), but
it has the advantage to be quite simple.

The main result of this section is the following

\begin{theorem} \label{thm_monotonic_convergence}
Suppose the data completion problem admits a unique solution $x_s$. Then $\Vert x_\ep - x_s \Vert_{A,b}$ is strictly increasing with respect to $\ep$.
\end{theorem}

We need to prove first the following two results, which are true whether or not the data completion problem admits a solution:

\begin{lemma} \label{lemma_sign_b}
For all $m \in \mathbb{N}$, for all $n \in \mathbb{N}$, $\displaystyle (-1)^{m+n} b(x_\ep^{(m)},x_\ep^{(n)})> 0$.
\end{lemma}

\begin{proof}
For $m\in \mathbb{N}$, let us define the axiom of induction:
$$
P(m):\quad \forall n\in \left\lbrace 0,\cdots, m \right\rbrace,\ (-1)^{m+n} b(x_\ep^{(m)},x_\ep^{(n)}) > 0.
$$
Obviously, $P(0)$ is true, as $y\neq 0 \Rightarrow x_\ep \neq 0$. 

Suppose $P(M)$ is true for some $M \in \mathbb{N}$. Let $k$ be in $\left\lbrace 0,\cdots , M+1 \right\rbrace$. 
\begin{itemize}
\item[$\bullet$] if $k = M+1$, then $\displaystyle (-1)^{2M+2} b(x_\ep^{(M+1)},x_\ep^{(M+1)})  = \|x_\ep^{(M+1)} \|_b^2> 0$ (as $y\neq 0$)
\item[$\bullet$] if $k = M$, then, by definition of $x_\ep^{(M+1)}$,
$$
(-1)^{2M+1} b(x_\ep^{(M+1)},x_\ep^{(M)}) = - b(x_\ep^{(M+1)},x_\ep^{(M)})
 = \frac{\|A x_\ep^{(M+1)} \|_\mathcal{Y}^2 + \ep \|x_\ep^{(M+1)} \|_b^2}{M+1} > 0
$$
\item[$\bullet$] if $k < M$, then, using successively the definition of $x_\ep^{(k+1)}$ and $x_\ep^{(M+1)}$, we obtain
\begin{align*}
b(x_\ep^{(M+1)},x_\ep^{(k)}) & = \frac{-1}{k+1}\Big((Ax_\ep^{(M+1)},A x_\ep^{(k+1)})_\mathcal{Y} 
+ \ep b(x_\ep^{(M+1)}, x_\ep^{(k+1)})  \Big) \\
 & = \frac{M+1}{k+1} (x_\ep^{(M)},x_\ep^{(k+1)})_b.
\end{align*}
As $k+1 \in \left\lbrace 0,\ldots , M \right\rbrace$, $P(M)$ implies 
$$
(-1)^{M+k+1} b(x_\ep^{(M)},x_\ep^{(k+1)}) > 0 \Rightarrow(-1)^{M+k+1} b(x_\ep^{(M+1)},x_\ep^{(k)})>0.
$$
\end{itemize}
\end{proof}

\begin{proposition}
The quantity $\|A x_\ep - y\|_\mathcal{Y}$ is a strictly increasing function of $\ep$.
\end{proposition}

\begin{proof}
Defining $\displaystyle g : \ep \in \mathbb{R}^+_* \mapsto \frac{1}{2} \|A x_\ep -y \|_\mathcal{Y}^2$, we have
$$
g'(\ep) = (Ax_\ep - y, x_\ep^{(1)})_\mathcal{Y} = -\ep b(x_\ep,x_\ep^{(1)}) >0.
$$
\end{proof}

\begin{proof}[Proof of theorem \ref{thm_monotonic_convergence}]
define $ \displaystyle g := \ep \in \mathbb{R}^+_*\mapsto \frac{1}{2} \|x_\ep - x_s \|_b^2$. We have $g'(\ep) = b(x_\ep - x_s,x_\ep^{(1)})$ and $g''(\ep) = \|x_\ep^{(1)} \|^2_b
 + b(x_\ep - x_s,x_\ep^{(2)})$. Therefore
 \begin{align*}
 \ep g''(\ep) & = \ep \|x_\ep^{(1)} \|_b^2 + \ep b(x_\ep - x_s,x_\ep^{(2)}) \\
  & = \ep \|x_\ep^{(1)} \|_b^2 - (A(x_\ep - x_s),Ax_\ep^{(2)})_\mathcal{Y} - 2 b(x_\ep - x_s,x_\ep^{(1)})
   \text{ ( definition of $x_\ep^{(2)}$)} \\
    & = \ep \|x_\ep^{(1)} \|_b^2 + \ep b(x_\ep,x_\ep^{(2)}) - 2 b(x_\ep - x_s,x_\ep^{(1)})
    \text{ ( definition of $x_\ep$ and $Ax_s = y$)} \\
    & = \ep \|x_\ep^{(1)} \|_b^2 + \ep b(x_\ep,x_\ep^{(2)})- 2 g'(\ep).
 \end{align*}
So $g$ verifies the following ODE: $\ep g''(\ep) + 2 g'(\ep) = \ep \|x_\ep^{(1)} \|_b^2 + \ep b(x_\ep,x_\ep^{(2)})$,
that is 
$$\frac{d}{d\ep} (\ep^2 g'(\ep)) =  \ep^2 \|x_\ep^{(1)} \|_b^2 + \ep^2 b(x_\ep,x_\ep^{(2)})>0.$$
Therefore, $\ep^2g'(\ep)$ is a strictly increasing function. As $\|x_\ep - x_s \|_b\leq \|x_s\|_b$
and $\|x_\ep^{(1)} \|_b \leq\displaystyle \frac{1}{\min(1,\ep)} \|x_s\|_b$, we have
$$
|\ep^2 g'(\ep)| = |\ep^2 b(x_\ep-x_s,x_\ep^{(1)})| \leq \ep^2 \|x_\ep - x_s\|_b \|x_\ep^{1}\|_b \leq \ep \|x_s\|_b
\xrightarrow[\ep \rightarrow 0]{} 0,
$$
which leads to $\ep^2 g'(\ep) >0 \Rightarrow g'(\ep) >0$, which ends the demonstration, as $\Vert x_\ep - x_s \Vert_{A,b}
 = \sqrt{ \Vert A x_\ep - y \Vert_\mathcal{Y}^2 + \Vert x_\ep - x_s \Vert_b^2} $.
\end{proof}

\section{Iterated quasi-reversibility}

As seen in the previous section, the quasi-reversibility method can be viewed as a Tykhonov regularization
of our abstract data completion problem. Therefore, it seems natural to study a well-known extension of such
regularization, namely the iterated Tykhonov regularization method, to our problem: we then obtain the
\textit{iterated quasi-reversibility method}.

The iterated quasi-reversibility method consists in solving iteratively  quasi-reversibility problems, each
one depending on the solution of the previous one. More precisely, we define a sequence of quasi-reversibility
solutions by induction : $X_\ep^{-1} = 0_\mathcal{X}$ and for all $M \in \mathbb{N}$, $X_\ep^{M} $ is the unique
element of $\mathcal{X}$ verifying
$$
(A X_\ep^M, A x)_\mathcal{Y} + \ep b(X_\ep^M,x) = (y,Ax)_\mathcal{Y} + \ep b(X_\ep^{M-1},x),\quad \forall x \in \mathcal{X}.
$$
It is not difficult to verify that the sequence is well-defined. In particular, it is clear that
$X_\ep^0 = x_\ep$, solution of the quasi-reversibility problem.

Our study of the iterated quasi-reversibility method is based on the following result, 
which highlighted  the  link between the solutions of the iterated quasi-reversibility method
$(X_\ep^M)_{M \in \left\lbrace -1 \right\rbrace\cup \mathbb{N}}$ and the derivatives of $x_\ep$
with respect to the parameter of regularization $\ep$:

\begin{theorem}
For all $\ep > 0$, for all $M \in \left\lbrace-1\right\rbrace \cup\mathbb{N}$, we have $$X_\ep^M  = \sum_{m=0}^M (-1)^m 
\frac{\ep^m}{m\,!} x_\ep^{(m)}.$$
\end{theorem}

\begin{proof}
Denote $\displaystyle \tilde{X}_\ep^{M} := \sum_{m=0}^M (-1)^m 
\frac{\ep^m}{m\,!} x_\ep^{(m)}$. For $M=-1$, the sum is empty, therefore we have $\tilde{X}_\ep^{-1} = 0_\mathcal{X} = X_{\ep}^{-1}$.
For $M = 0$, we also have $\tilde{X}_\ep^0 = x_\ep = X_\ep^0$. Finally, for $M \geq 1$ and $1\leq m \leq M+1$, in virtue of the  definition of 
$x_\ep^{(m)}$, we have for all $x \in \mathcal{X}$
\begin{align*}
\Big( A \left( (-1)^m \frac{\ep^m}{m!} x_\ep^{(m)}\right) ,A x \Big)_\mathcal{Y}  &+ \ep\ b\Big( \left( (-1)^m \frac{\ep^m}{m!} x_\ep^{(m)}\right)
, x \Big)\\
  &\quad \quad = \ep\ b \Big(\left( (-1)^{m-1} \frac{\ep^{m-1}}{(m-1)!} x_\ep^{(m-1)}\right) , x \Big).
\end{align*}
Summing for $m=1$ to $M+1$, and adding the equation verified by $x_\ep^{(0)} = x_\ep$, we obtain
$$
(A \tilde{X}_\ep^{M+1}, Ax)_\mathcal{Y} + \ep b (\tilde{X}_\ep^{M+1},x) = (f,Ax)_\mathcal{Y} + \ep b( \tilde{X}_\ep^M,x).
$$
A straightforward induction ends the proof.
\end{proof}

From now on, we suppose $y \neq 0$: if not, we have $X_\ep^M = 0$ for all $\ep$
and $M$.

\subsection{Some estimates on $X_\ep^M$ and $A X_\ep^M$}

We start with estimates on the $M$-th iterated quasi-reversibility solution, valid for any data $y$, admissible or not. In other words,
these estimates are valid whether or not the data completion problem has a solution. 

\begin{proposition} \label{prop_allestimXem}
For all $\ep > 0$, for all $M \in \mathbb{N}$, we have
\begin{itemize}
\item[\textnormal{(a)}] $\displaystyle \Vert X_\ep^{M-1} \Vert_b < \Vert X_\ep^{M} \Vert_b $
\item[\textnormal{(b)}] $\displaystyle \|A X_\ep^M \|_\mathcal{Y} < \| y\|_\mathcal{Y}$
\item[\textnormal{(c)}]$\displaystyle \|A X_\ep^M-y  \|_\mathcal{Y} < \| y\|_\mathcal{Y} $
\item[\textnormal{(d)}] $ \displaystyle \|A X_\ep^{M} - y \|_\mathcal{Y} < \|AX_\ep^{M-1} -y \|_\mathcal{Y}$.
\end{itemize}
\end{proposition}

\begin{proof}
We start with estimate (a): as $y \neq 0$, we have $0 = \Vert X_\ep^{-1} \Vert_b < \Vert x_\ep \Vert_b = \Vert X_\ep^1\Vert_b $.
Furthermore, for $M \in \mathbb{N}$, $\displaystyle \|X_\ep^M\|_b^2 = \sum_{k=0}^M \sum_{m=0}^M
 (-1)^{k+m} \frac{\ep^{k+m}}{k!\, m!} b(x_\ep^{(k)},x_\ep^{(m)})$. Therefore, from lemma \ref{lemma_sign_b} we obtain
$$
 \|X_\ep^{M+1}\|_b^2 - \|X_\ep^M\|_b^2  = 2 \sum_{m=0}^{M+1} (-1)^{M+1+m} \frac{\ep^{M+1+m}}{(M+1)!\, m!}
 b(x_\ep^{(M+1)},x_\ep^m) >0.
$$
Regarding estimates (b) and (c), we note that they hold for $M = 0$. Furthermore,
$$
\|AX_\ep^{M+1}\|_\mathcal{Y}^2   = (y,A X_\ep^{M+1})_\mathcal{Y} + \ep b(X_\ep^M,X_\ep^{M+1}) - \ep  \|X_\ep^{M+1} \|_b^2.
$$
Estimate (a) implies $b(X_\ep^{M+1},X_\ep^M) < \|X_\ep^{M+1} \|_b^2$, and
$$
\|AX_\ep^{M+1}\|_\mathcal{Y}^2   < (y,A X_\ep^{M+1})_\mathcal{Y}.
$$
Cauchy-Schwarz inequality implies then the estimate (b). Furthermore,
$$
\|A X_\ep^{M+1} - y \|_\mathcal{Y}^2 = (AX_\ep^{M+1} - y ,A X_\ep^{M+1})_\mathcal{Y} - (AX_\ep^{M+1} - y ,y)_\mathcal{Y} <
 - (AX_\ep^{M+1} - y ,y)_\mathcal{Y}
$$ 
which leads to  estimate (c). 

Finally, the case $M = 0$ of estimate (d) correspond to estimate (c) with same $M$.
For $M\in \mathbb{N}$, we note that
\begin{align*}
\|A X_\ep^{M+1}-y \|_\mathcal{Y}^2  & = \|\sum_{m=0}^{M+1} (-1)^m \frac{\ep^m}{m!} A x_\ep^{(m)} - y \|_\mathcal{Y}^2 \\
 & = \|A X_\ep^{M}-y \|_\mathcal{Y}^2 +  (-1)^{M+1} \frac{2\,\ep^{M+1}}{(M+1)!} (A x_\ep^{(M+1)},AX_\ep^M - y)_\mathcal{Y}.
\end{align*}
Therefore, it is sufficient to determine the sign of $(-1)^{M+1}(A x_\ep^{(M+1)},AX_\ep^M - y)_\mathcal{Y}$. By definition, we have
$$
(Ax_\ep^{(M+1)},Ax_\ep - y)_\mathcal{Y} = -\ep\ b(x_\ep^{(M+1)},x_\ep)
$$
and for all $m \in \left\lbrace 1,\ldots, M \right\rbrace$, 
\begin{align*}
\Big( A x_\ep^{(M+1)}, A \left( (-1)^m \frac{\ep^m}{m!}  x_\ep^{(m)}\right) \Big)_\mathcal{Y}
 & =  (-1)^{m+1} \frac{\ep^{m+1}}{m!}\ b( x_\ep^{(M+1)} ,  x_\ep^{(m)}  )  \\
& \quad \quad   +  (-1)^{m+1} \frac{\ep^m}{(m-1)!}\ b (x_\ep^{(M+1)}, x_\ep^{(m-1)} ).
\end{align*}
Summing these equations for $m=0$ to $M$, we obtain 
$$
(Ax_\ep^{(M+1)},AX_\ep^M -y)_\mathcal{Y}   =  (-1)^{M+1} \frac{\ep^{M+1}}{M!} \ b(x_\ep^{(M+1)},x_\ep^{(M)}).
$$
In conclusion, $(-1)^{M+1} (Ax_\ep^{(M+1)},AX_\ep^M -y)_\mathcal{Y} =  \frac{\ep^{M+1}}{M!} b(x_\ep^{(M+1)},x_\ep^{(M)}) <0$ by
lemma \ref{lemma_sign_b}.
The result follows.
\end{proof}

\begin{proposition} \label{prop_bound_XepM}
For all $\ep>0$, for all $M\in \left\lbrace - 1 \right\rbrace \cup\mathbb{N}$, $\|X_\ep^M \|_b \leq \sqrt{\frac{2(M+1)}{\ep}} \|y\|_\mathcal{Y}$.
\end{proposition}

\begin{proof} Proposition \ref{prop_bound_XepM} is obviously true for $M = -1$.  Let $M \in \mathbb{N}$. We consider first the  inductive sequence: 
$$
\left\lbrace
\begin{array}{l}
x_1 > 0 \\
\forall M \in \mathbb{N},\ x_{M+1} > 0 \text{ and } x_{M+1}^2 - x_{M+1}x_M - x_1^2 = 0. 
\end{array}
\right.
$$
Note that the sequence is well defined as $p_M(x) := x^2 - x_M\, x -x_1^2$ verifies $p_M(0) <0$ and therefore
has a unique strictly positive root. 

We prove by induction that $x_M < \sqrt{2\,M} x_1$. 
It obviously holds for $M=1$. Suppose that $x_M < \sqrt{2\,M}x_1$ for some $M\in \mathbb{N}$. Then
\begin{align*}
p_M\left(\sqrt{2(M+1)}x_1\right) & = 2(M+1)x_1^2 - \sqrt{2(M+1)}x_1 x_M - x_1^2 \\
& > (2M+1)x_1^2 - 2 \sqrt{M+1}\sqrt{M}x_1^2  \\
 & = \left( \sqrt{M+1} - \sqrt{M} \right)^2 x_1^2 > 0
\end{align*}
and therefore $x_{M+1} < \sqrt{2(M+1)}x_1$.

Now, we specify the sequence, defining $x_1 := \displaystyle\frac{1}{\sqrt{\ep}} \|y\|_\mathcal{Y}$, and prove by induction that
$\|X_\ep^M\|_b \leq x_{M+1}$. Suppose it holds for some $M \in \mathbb{N}$. Note that by definition of $X_\ep^{M+1}$, we have
\begin{align*}
\|A X_\ep^{M+1} \|_\mathcal{Y}^2 + \ep \|X_\ep^{M+1}\|_b^2 & = (y,AX_\ep^{M+1})_\mathcal{Y} + \ep b(X_\ep^M,X_\ep^{M+1})\\
 & \leq \|y\|_\mathcal{Y} \|AX_\ep^{M+1}\|_\mathcal{Y} + \ep \|X_\ep^M\|_b \|X_\ep^{M+1}\|_b \\
 & \leq \|y\|_\mathcal{Y}^2 + \ep\, x_{M+1} \|X_\ep^{M+1}\|_b,   
\end{align*}
(we used estimate (b) of proposition \ref{prop_allestimXem} here) which in particular implies
$$
\|X_\ep^{M+1}\|_b^2 - x_{M+1} \|X_\ep^{M+1}\|_b - x_1^2 <0.
$$
The definition of the sequence $(x_M)_{M\in\mathbb{N}}$ implies directly $\|X_\ep^{M+1}\|_b \leq x_{M+2}$. As the result is true for $M =0$, 
the proposition follows.
\end{proof}

\begin{remark}
Actually, for $M\in \mathbb{N}$, the inequality in proposition \ref{prop_bound_XepM} is strict.
\end{remark}

\subsection{The case of exact data}

From now on, we suppose that $y \in Y_{adm}$, and denote $x_s$ the solution of the abstract data completion problem. 
We define $R_\ep^M := X_\ep^M - x_s$, the discrepancy between the $M$-th iterated QR solution, and the exact solution.
Note that by definition, $R_\ep^{-1} = -x_s$ and for all $M \in \mathbb{N}$,
\begin{equation} \label{eqn_RN}
(AR_\ep^M, Ax)_\mathcal{Y} + \ep\, b (R_\ep^M,x)= \ep\, b(R_\ep^{M-1},x),\quad \forall x \in \mathcal{X}.
\end{equation}
We aim to prove the following theorem:

\begin{theorem}  \label{thm_proof_conv_RepM}
For all  $\ep > 0$, $\displaystyle R_\ep^M \xrightarrow[\mathcal{X}]{M \rightarrow \infty} 0$. In other words, for any  $\ep >0$, $X_\ep^M$ converges to $x_s$ as $M$ goes to infinity.
\end{theorem}

As $X_\ep^M = \sum_{m=0}^M (-1)^m \frac{\ep^m}{m\,!}x_\ep^{(m)}$, it means that the sum converges as $M$ goes to infinity. 
In other words, it means that if it exists $x_s \in \mathcal{X}$ solution of $A x_s = y$, then
$$
x_s = \sum_{m=0}^\infty  (-1)^m \frac{\ep^m}{m\,!}x_\ep^{(m)},
$$
hence the solution of the data completion problem can be seen as a series of derivatives of the quasi-reversibility solution w.r.t. the parameter
$\ep$.

Let $\ep> 0$ be fixed. We start with the following estimates

\begin{proposition} \label{prop_XepM_bnd}
For all $\ep > 0$, for all $M\in \left\lbrace-1 \right\rbrace \cup \mathbb{N}$, 
\begin{itemize}
\item[-] $\|R_\ep^{M+1} \|_b < \|R_\ep^M\|_b$
\item[-] $\Vert X_\ep^M \Vert_b \leq \Vert x_s \Vert_b$.
\end{itemize}
As a consequence, we have $\Vert R_\ep^M \Vert_b \leq \Vert x_s \Vert_b$ for all $M \in \left\lbrace -1 \right\rbrace \cup\mathbb{N}$ and all $\ep > 0$.
\end{proposition}

\begin{proof}
Choosing $x = R_\ep^{M}$ in (\ref{eqn_RN}), we obtain
$$
\|A R_\ep^{M} \|_\mathcal{Y}^2 + \ep \|R_\ep^{M} \|_b^2 = \ep b(R_\ep^{M-1},R_\ep^{M}) \Rightarrow
 \|R_\ep^{M} \|_b^2 < b(R_\ep^{M},R_\ep^{M-1}) \leq \|R_\ep^M \|_b \|R_\ep^{M-1}\|_b,
$$
hence the first estimate is valid. Note that in particular, 
 as $\Vert R_\ep^M \Vert_b \leq \Vert R_\ep^0\Vert_b
 = \Vert x_\ep - x_s \Vert_b$, we have $\Vert R_\ep^M\Vert_b \xrightarrow[]{\ep \rightarrow 0} 0$ for any $M \in \mathbb{N}$. 
 
 Let us now focus on the second estimate, which is directly true for $M=-1$.
 For $M\in \mathbb{N}$,
let us define $g := \ep \in \mathbb{R}^+_* \mapsto \frac{1}{2}\|X_\ep^M \|_b^2$. As 
by definition, $X_\ep^M = \sum_{m=0}^M (-1)^m \frac{\ep^m}{m!} x_\ep^{(m)}$, we have
$$
\frac{d}{d\ep} X_\ep^M = \sum_{m=0}^M (-1)^m \frac{\ep^m}{m!} x_\ep^{(m+1)} + \sum_{m=1}^M (-1)^m \frac{\ep^{m-1}}{(m-1)!} x_\ep^{(m)} = (-1)^M \frac{\ep^M}{M!} x_\ep^{(M+1)}.
$$
Therefore, we have
$$
g'(\ep) = b(\frac{d}{d\ep}X_\ep^M,X_\ep^M) = \sum_{m=0}^M (-1)^{M+m} \frac{\ep^{M+m}}{M!\,m!} b(x_\ep^{(M+1)},x_\ep^{(m)}) <0,
$$
implying in particular that $\|X_\ep^M\|_b \leq \displaystyle \lim_{\eta \rightarrow 0} \|X_\eta^M \|_b = \|x_s\|_b$, as 
$\Vert R_\ep^M\Vert_b \xrightarrow[]{\ep \rightarrow 0} 0$.
\end{proof}

\begin{proposition} \label{prop_conv_ARM}
The series $\displaystyle \sum_M \|A R_\ep^M \|_\mathcal{Y}^2$ converges,  therefore $$\displaystyle \lim_{M\rightarrow \infty } \|A R_\ep^M \|_\mathcal{Y} = 0.$$
\end{proposition}

\begin{proof}
For any $M \in \mathbb{N}$, we have
$$
\|A R_\ep^{M+1} \|_\mathcal{Y}^2 + \ep \|R_\ep^{M+1} \|_b^2 = \ep b(R_\ep^{M},R_\ep^{M+1})
$$
which leads to
$$
\|A R_\ep^{M+1} \|_\mathcal{Y}^2 < \ep b(R_\ep^{M},R_\ep^{M+1}) \leq \ep \|R_\ep^M \|_b \|R_\ep^{M+1}\|_b
 < \ep \|R_\ep^M \|_b^2 = \ep b(R_\ep^{M-1},R_\ep^M) - \|A R_\ep^M \|_\mathcal{Y}^2.
$$
Therefore, we obtain
$$
\|A R_\ep^{M+1} \|_\mathcal{Y}^2 + \|A R_\ep^M \|_\mathcal{Y}^2 < \ep b(R_\ep^{M-1},R_\ep^M)
$$
and by an immediate induction
$$
\sum_{m=1}^M \|AR_\ep^m \|_\mathcal{Y}^2 \leq \ep \|R_\ep^0 \|^2_b \leq \ep \|x_s \|^2_b.
$$
Therefore , the series $\displaystyle \sum \|AR_\ep^m \|_\mathcal{Y}^2 $ converges. The property follows.
\end{proof}

Theorem  \ref{thm_proof_conv_RepM} follows from the previous proposition: indeed,
let $\varphi :\mathbb{N} \rightarrow \mathbb{N}$ be a strictly increasing map, and define $\tilde{R}_\ep^M := R_\ep^{\varphi(M)}$. 
As
$$
\Vert \tilde{R}_\ep^M \Vert_{A,b}^2 = \Vert A \tilde{R}_\ep^M \Vert_\mathcal{Y}^2 + \Vert \tilde{R}_\ep^M \Vert_b^2 \leq \Vert y\Vert_\mathcal{Y}^2+
\Vert x_s \Vert_b^2
$$
we have  that 
$(\tilde{R}_\ep^M)_{M\in\mathbb{N}}$ is a bounded sequence in $\mathcal{X}$. Consequently, there exists $\vartheta : \mathbb{N} \rightarrow\mathbb{N}$, a strictly increasing map such that $\hat{R}_\ep^M := \tilde{R}_\ep^{\vartheta(M)}$ weakly
converges to $R^\infty$ in $\mathcal{X}$. As $A$ is linear and continuous, we directly obtain from proposition \ref{prop_conv_ARM} that $A R^\infty = 0_Y$,
which implies $R^\infty = 0_\mathcal{X}$ as $A$ is one-to-one.

We hence have obtained that $\hat{R}_\ep^M \underset{\scriptscriptstyle M \rightarrow \infty}{\rightharpoonup} 0_\mathcal{X}$, or equivalently $\hat{X}_\ep^M \underset{\scriptscriptstyle M \rightarrow \infty}{\rightharpoonup} x_s$.
But we know from propositions \ref{prop_allestimXem} and \ref{prop_XepM_bnd} that $\|\hat{X}_\ep^M \|_{A,b} \leq \|x_s \|_{A,b}$, implying that $\hat{X}_\ep^M$ strongly converges to $x_s$, and it is then not difficult to show that the whole sequence 
$X_\ep^M$ strongly converges to $x_s$ as $M$ goes to infinity.

\subsection{The case of noisy data} \label{section_iterated_noisy_data}

In this section, we suppose that our exact data, denoted $y_{ex} \in \mathcal{Y}$, for which the data completion problem admits a unique solution
$x_s \in \mathcal{X}$, is corrupted by some noise. The obtained perturbed data, denoted $y^\delta \in \mathcal{Y}$, is supposed to verify
$\Vert y^\delta - y_{ex} \Vert_\mathcal{Y} \leq \delta$: in other words, we know the amplitude of noise on the data. On the other hand, there might
or might not be $x\in \mathcal{X}$ such that $A x = y^\delta$: we don't know if $y^\delta$ is an admissible solution or not.

From now on, for any $y \in \mathcal{Y}$, we will denote $X_\ep^M( y)$ the M-th iterated quasi-reversibility solution with $y$ as
data. Our main objective in this section is to propose an admissible strategy  to choose $M$ as a function of $\delta$, the amplitude of noise,
to ensure that, when $\delta$ goes to zero, $X_\ep^{M(\delta)}$ tends to the exact solution $x_s$. As pointed out in proposition \ref{prop_unstability}
and remark \ref{rk_unstability}, this is a crucial point in the study of data completion problems. 

A first important remark is the following: $A X_\ep^M(y)$ always converges to $y$, regardless of the admissibility of $y$ as data for the
data completion problem.

\begin{proposition} \label{prop_conv_general_AXepM}
For any $y \in \mathcal{Y}$, for any $\ep > 0$, $A X_\ep^M(y) \xrightarrow[M \rightarrow \infty]{\mathcal{Y}} y$. 
\end{proposition}

\begin{proof}
As $\mathcal{Y}_{adm}$ is dense in $\mathcal{Y}$, for any $\eta > 0$, it exists $y_\eta \in \mathcal{Y}_{adm}$ such that $\Vert  y_\eta - y\Vert_\mathcal{Y}\leq \frac{\eta}{3}$.
Proposition \ref{prop_allestimXem} (b) implies $\Vert A X_\ep^M(y) - A X_\ep^M(y_\eta) \Vert_\mathcal{Y} \leq \Vert y - y_\eta \Vert_\mathcal{Y} \leq \frac{\eta}{3} $.
Finally, as $y_\eta \in \mathcal{Y}_{adm}$, there exists $M_\eta > 0$ such that for any $M \geq M_\eta$, 
$
\Vert A X_\ep^M(y_\eta) - y_\eta \Vert_\mathcal{Y} \leq \frac{\eta}{3}.
$
The result follows.
\end{proof}

Next proposition defines the admissible choices of $M$ to ensure the desired convergence:

\begin{proposition} \label{prop_adm_str_delta}
For any $\ep>0$, for any choice of $M:=M(\delta)$ such that $M(\delta) \xrightarrow[\delta \rightarrow 0]{} +\infty$
and $\delta \sqrt{M(\delta)} \xrightarrow[\delta \rightarrow 0]{} 0$, we have $X_\ep^{M(\delta)}(y^\delta) \xrightarrow[\delta \rightarrow 0]{\mathcal{X}} x_s$.
\end{proposition}

\begin{proof}
obviously, we have, for any $\ep>0$ and $M \in \mathbb{N}$
$$
\|X_\ep^M(y^\delta) - x_s \|_{A,b} \leq \|X_\ep^M(y^\delta)- X_\ep^M(y) \|_{A,b} + \|X_\ep^M(y) - x_s \|_{A,b}.
$$
If $M:= M(\delta)$ verifies $\lim_{\delta \rightarrow 0} M(\delta) = +\infty$, theorem \ref{thm_proof_conv_RepM}
implies directly that $\|X_\ep^{M(\delta)}(y) - x_s \|_{A,b} \xrightarrow[\delta \rightarrow 0]{} 0$. Furthermore, 
propositions \ref{prop_allestimXem} and \ref{prop_XepM_bnd} imply
$$
\|X_\ep^M(y^\delta)- X_\ep^M(y) \|_{A,b}^2 \leq \Big(1 + \frac{2(M+1)}{\ep} \Big)\|y^\delta - y \|_\mathcal{Y}^2 = \Big(1 + \frac{2(M+1)}{\ep} \Big)\, \delta^2.
$$ 
The result follows.
\end{proof}

Proposition \ref{prop_adm_str_delta} defines the admissible strategies to choose $M$ depending on $\delta$. An admissible choice
could be $M(\delta) := \left\lfloor \frac{1}{\sqrt{\delta}} \right\rfloor $ for example.  But such a choice, if it guarantees the convergence 
of the method, does not correspond to any precise objective. We therefore focus on another method to choose $M (\delta)$.

Let $r>1$.
For a fixed $\ep>0$, we define $\mathbb{M}_\delta := \left\lbrace M \in \mathbb{N},\ \|AX_\ep^M(y^\delta) - y^\delta \|_\mathcal{Y} \leq r\delta \right\rbrace$. 
Proposition \ref{prop_conv_general_AXepM} implies that $\mathbb{M}_\delta$ is non-empty, and we define $M(\delta)$ as the minimum element
of $\mathbb{M}_\delta$: $M(\delta) := \min \left\lbrace M \in \mathbb{M}(\delta) \right\rbrace$. 

$M(\delta)$ is chosen accordingly to the \textit{Morozov discrepancy principle:} it is the first $M \in \mathbb{N}$ such that 
the distance between $A X_\ep ^M$ and $y^\delta$ is (approximately) equal to the distance between $A x_s = y$ and
$y^\delta$:
$\Vert A X_\ep^M - y^\delta \Vert_{\mathcal{Y}} \approx \Vert A x_s - y^\delta \Vert_{\mathcal{Y}} $.
This method to choose $M$ depending on $\delta$ has two interesting characteristics:
\begin{enumerate}
\item with this choice, one does the minimum number of iterations of the iterated quasi-reversibility  method required to obtain an error in the residual $\Vert A X_\ep^M - y^\delta \Vert_\mathcal{Y}$
of same order of the error on the data.
\item such choice is admissible, in the sense of proposition \ref{prop_adm_str_delta}.
\end{enumerate}

We  now prove that $M(\delta)$ is
an admissible choice.

\begin{proposition}
$M(\delta) \xrightarrow[]{\delta \rightarrow 0} + \infty$.
\end{proposition}

\begin{proof}
Suppose it is not the case. Then there exists a sequence of strictly positive real numbers $\delta_n$ and a positive constant $C$ such that
$\delta_n \xrightarrow[]{n \rightarrow \infty }0$ and $M(\delta_n) \leq C$. It implies the existence of a subsequence (still denoted
$\delta_n$) and $M_\infty \in \mathbb{N}$ such that $\delta_n \xrightarrow[]{n \rightarrow \infty }0$
and $M(\delta_n) \xrightarrow[]{n \rightarrow \infty} M_\infty$. In particular, it exists $N\in \mathbb{N}$ such that for all $n \geq N$,
$M(\delta_n) = M_\infty$.

For $n \geq N$,
the definition of $M(\delta)$ implies
$$
\Vert A X_\ep^{M_\infty}(y^{\delta_n})  - y_{ex} \Vert_\mathcal{Y} \leq \Vert A X_\ep^{M_\infty}(y^\delta_n)  - y^{\delta_n} \Vert_\mathcal{Y}
+ \Vert y^{\delta_n} - y_{ex} \Vert_\mathcal{Y} \leq (r+1) \delta_n \xrightarrow[]{n \rightarrow \infty} 0.
$$
Consequently, using proposition \ref{prop_allestimXem} we have
\begin{align*}
\|A X_\ep^{M_\infty}(y_{ex}) - y_{ex} \|_\mathcal{Y} &\leq \|A X_\ep^{M_\infty}(y_{ex}) - A X_\ep^{M_\infty}(y^{\delta_n}) \|_\mathcal{Y}
 + \|A X_\ep^{M_\infty}(y^{\delta_n}) - y_{ex} \|_\mathcal{Y}  \\
  & \leq  \Vert  y_{ex} - y^{\delta_n} \Vert _\mathcal{Y} + (r+1) \delta_n = (r+2) \delta_n \xrightarrow[n \rightarrow \infty]{} 0,
 \end{align*}
that is $AX_\ep^{M_\infty}(y_{ex}) = y_{ex}$. If $M_\infty = -1$, we directly obtain $y_{ex} = 0$, in contradiction with the hypothesis.
If $M_\infty = 0$, we obtain $x_\ep = x_s$, which again (corollary \ref{coro_1}) implies $y_s = 0$
Finally, if $M_\infty > 0$, as for all $x\in \mathcal{X}$, 
$$
(AX_\ep^{M_\infty}(y_{ex}),Ax)_\mathcal{Y} + \ep b (X_\ep^{M_\infty}(y_{ex}),x) = \ep b(X_\ep^{M_\infty-1}(y_{ex}),x) + (y_{ex},Ax)_\mathcal{Y},
$$
we obtain $X_\ep^{M_\infty}(y_{ex})= X_\ep^{M_\infty-1}(y_{ex})$, or equivalently $x_\ep^{(M_\infty)}(y_{ex}) = 0$, which again implies $y_{ex} = 0$
by corollary \ref{coro_2}. We obtain once again a contradiction, which ends the proof.
\end{proof}

\begin{proposition}
$\displaystyle \lim_{\delta \rightarrow 0} \delta \sqrt{M(\delta)} = 0$.
\end{proposition}

\begin{proof}
by definition, we have $\|A X_\ep^{M(\delta) - 1}(y^\delta) - 
y^\delta \|_\mathcal{Y}> r\delta$. Therefore
\begin{align*}
\|A X_\ep^{M(\delta) -1}(y_{ex}) - y_{ex} \|_\mathcal{Y} & = \| A X_\ep^{M(\delta) -1}(y_{ex}) - A X_\ep^{M(\delta)-1}(y^\delta)
+ A X_\ep^{M(\delta)-1}(y^\delta) - y^\delta  + y^\delta - y_{ex} \|_\mathcal{Y} \\
 & \geq \| A X_\ep^{M(\delta)-1}(y^\delta) - y^\delta \|_\mathcal{Y} - \|A X_\ep^{M(\delta) -1}(y_{ex}) - A X_\ep^{M(\delta)-1}(y^\delta) - (y_{ex} - y^\delta)  \|_\mathcal{Y} \\
  & > (r-1) \delta
\end{align*}
as proposition \ref{prop_allestimXem}, estimate (c) implies 
$$
\|A X_\ep^{M(\delta) -1}(y_{ex}) - A X_\ep^{M(\delta)-1}(y^\delta) - (y_{ex} - y^\delta)  \|_\mathcal{Y}
\leq \|y_{ex} - y^\delta \|_\mathcal{Y} \leq \delta.
$$
We hence have obtained:
$$
(M(\delta)-1) \delta^2(r-1)^2 \leq (M(\delta)-1) \|A X_\ep^{M(\delta) -1}(y_{ex}) - y_{ex} \|_\mathcal{Y}^2 = 
(M(\delta)-1) \|A R_\ep^{M(\delta) -1}(y_{ex}) \|_\mathcal{Y}^2. 
$$
As $\|AR_\ep^{m+1}(y_{ex})\|_\mathcal{Y} < \|AR_\ep^m(y_{ex}) \|_\mathcal{Y}$ and  $\sum_m \|A R_\ep^m(y_{ex}) \|_\mathcal{Y}^2$ converges, 
$ m \|AR_\ep^m(y_{ex})\|_\mathcal{Y}^2 $ tends to zero as $m$ goes to infinity. 
Therefore, $(M(\delta)-1) \|A R_\ep^{M(\delta) -1}(y_{ex}) \|_\mathcal{Y}^2 $ goes to zero as $\delta$ tends to zero, implying that
$$
 \lim_{\delta \rightarrow 0} (M(\delta)-1) \delta^2(r-1)^2 = 0.
$$
The result follows.
\end{proof}

\section{Quasi-reversibility methods for data completion problems for the Poisson's equation and the heat equation} \label{sect_QRconcrete}

We will now go back to the data completion problems described in the introduction, and verify that they correspond to the abstract setting
introduced in section 2.

\subsection{Poisson's equation}

As mentioned in the introduction, the data completion problem for Poisson's equation is: for $(f,g_D,g_N) \in L^2(\Omega) \times L^2(\Gamma)
\times L^2(\Gamma)$, find $u \in H^1(\Omega)$ s.t.
$$
\left\lbrace
\begin{array}{ccl}
- \nabla \cdot \sigma \nabla u &=& f \text{ in } \Omega \\
 u &= &g_D \text{ on } \Gamma \\
 \sigma \nabla u \cdot \nu &= &g_N \text{ on } \Gamma.
\end{array}
\right.
$$
We could directly use this formulation of the problem to obtain a quasi-reversibility regularization. However, if we do so, we obtain a fourth-order
variational problem, which is rather difficult to discretize as we would need $C^1$ or non-conforming finite elements which are seldom available
in numerical solvers. Therefore, we first modify the problem by introducing the flux $\mathbf{p} := \sigma \cdot \nabla u$ as an additional unknown,
following the idea introduced in \cite{Dar2}.  It verifies $- \nabla \cdot \mathbf{p} = -\nabla \cdot \sigma \nabla u = f \in L^2(\Omega)$
and $\mathbf{p} \cdot \nu = \sigma \nabla u \cdot \nu = g_N \in L^2(\Gamma)$, hence 
$$
\mathbf{p} \in  \tilde{H}_{\mathsf{div}}(\Omega)
 := \left\lbrace \mathbf{q} \in L^2(\Omega)^d,\ \nabla \cdot \mathbf{q} \in L^2(\Omega), \ \mathbf{q} \cdot \nu \in L^2(\Gamma) \right\rbrace.
$$
$\tilde{H}_{\mathsf{div}}(\Omega)$, endowed with the scalar product 
$$
(\mathbf{p}, \mathbf{q})_{\tilde{H}_{\mathsf{div}}} := \int_\Omega \Big( \mathbf{p}\cdot \mathbf{q} + (\nabla \cdot \mathbf{p})\
( \nabla \cdot \mathbf{q} ) \Big) \, dx + \int_\Gamma (\mathbf{p}\cdot \nu) \, (\mathbf{q}\cdot \nu)\, dS
$$
is an Hilbert space \cite{Fernand}. We modify the data completion problem the following way:

\begin{problem}
For $f$, $g_D$ and $g_N$ in respectively $L^2(\Omega)$, $L^2(\Gamma)$ and $L^2(\Gamma)$, 
find $(u,\mathbf{p}) \in H^1(\Omega) \times \tilde{H}_{\mathsf{div}}(\Omega)$ such that 
$$
\left\lbrace
\begin{array}{rcl}
\sigma \nabla u & = & \mathbf{p} \text{ in } \Omega \\
- \nabla \cdot \mathbf{p} & = & f \text{ in } \Omega \\
u & = & g_D \text{ on } \Gamma \\
\mathbf{p}\cdot  \nu & = & g_N \text{ on } \Gamma .
\end{array}
\right.
$$
\end{problem}
Obviously, this is exactly the same problem as previously. However, this small modification will lead to a second-order variational
quasi-reversibility regularization in the product space $H^1\times \tilde{H}_{\mathsf{div}}$, easily discretized using standard finite-elements.

To fit in our abstract setting, we introduce the operator 
\begin{align*}
A : & (u,\mathbf{p})  \in \mathcal{X} = H^1(\Omega) \times  \tilde{H}_{\mathsf{div}}(\Omega) \\
  & \mapsto (\sigma\nabla u - \mathbf{p}, -\nabla \cdot \mathbf{p}, u_{\vert \Gamma}, \mathbf{p}\cdot \nu_{\vert \Gamma})  \in
\mathcal{Y} = L^2(\Omega)^d \times L^2(\Omega) \times L^2(\Gamma) \times L^2(\Gamma).
\end{align*}
The spaces $\mathcal{X}$ and $\mathcal{Y}$, endowed respectively with the scalar products
$$
\Big((u,\mathbf{p}), (v,\mathbf{q}) \Big)_\mathcal{X} := (u,v)_{H^1} + (\mathbf{p},\mathbf{q})_{\tilde{H}_{\mathsf{div}}}
$$
and
$$
\Big( (\mathbf{F}, f, g, h), (\tilde{\mathbf{F}},\tilde{f}, \tilde{g},\tilde{h} )\Big)_\mathcal{Y} :=
\int_\Omega \Big( \mathbf{F} \cdot \tilde{\mathbf{F}} + f\, \tilde{f} \Big) dx + \int_\Gamma \Big(  g\, \tilde{g} + h\, \tilde{h}\Big) \, ds
$$
are obviously Hilbert spaces, and
the data completion problems can be rewritten: find $(u,\mathbf{p})\in \mathcal{X}$ such that
$A (u,\mathbf{p}) = (\mathbf{0},f,g_D,g_N) \in \mathcal{Y}$. 

\begin{proposition}
$A$ is  linear, continuous, one-to-one. It is  not onto but has dense range. Additionally, $A$ is not a compact operator.
\end{proposition}

\begin{proof}
Clearly, $A$ is linear continuous. As the data completion problem for Poisson's equation is known to admits at most
a solution, but may have no solution, $A$ is one-to-one but not onto. Let us prove that $\overline{\textsf{Im}(A)}^\mathcal{Y} = \mathcal{Y}$. Let $(\mathbf{F},f,g,h) \in \mathcal{Y}$ such that
$$
\Big( A(u,\mathbf{p}),(\mathbf{F},f,g,h) \Big)_\mathcal{Y} = 0,\quad \forall (u,\mathbf{p}) \in \mathcal{Y},
$$
that is
$$
\int_\Omega \Big( (\sigma \nabla u - \mathbf{p}) \cdot \mathbf{F}  - (\nabla\cdot \mathbf{p})\, f\Big)\, dx + \int_\Gamma \Big( u \, g + (\mathbf{p}\cdot \nu)\, h\Big)\, ds = 0,\quad \forall (u,\mathbf{p}) \in H^1(\Omega) \times \tilde{H}_{\mathsf{div}}(\Omega).
$$
Choosing $u = \varphi \in C^\infty_c(\Omega)$ and $\mathbf{p} = \mathbf{0}$, we
obtain $- \nabla \cdot\sigma^T \mathbf{F} = 0 $, and in particular $\sigma^T \mathbf{F} \in H_{\mathsf{div}}(\Omega)$.
Choosing $u = 0$ and $\mathbf{p} = \mathbf{\Psi} \in C^\infty_c(\Omega)^d$, we obtain $\nabla f = \mathbf{F}$. Therefore,
 $f \in H^1(\Omega)$, and verifies $- \nabla \cdot \sigma^T \nabla f = 0$. Hence, taking $u \in H^1(\Omega)$ and $\mathbf{p} = 0$,
and using the Green formula, we obtain
$$
\langle \sigma^T \nabla f \cdot \nu, u\rangle = \int_\Omega \sigma^T \nabla f \cdot \nabla u  \, dx = \int_\Omega \sigma \nabla u \cdot
\nabla f \, dx = - \int_\Gamma u\, g \, ds,\quad \forall u \in H^1(\Omega),
$$
implying that $\sigma^T \nabla f \cdot \nu = - g$ on $\Gamma$ and $0$ on $\Gamma_c$. Taking $u =0$ and $\mathbf{p} \in \tilde{H}_{\mathsf{div}}(\Omega)$, and using the divergence theorem, we obtain
$$
\langle \mathbf{p}\cdot \nu , f \rangle = \int_\Omega \Big( \mathbf{p} \cdot \nabla f + \nabla\cdot \mathbf{p} \, f\Big)\, dx
= \int_\Gamma (\mathbf{p}\cdot\nu )\, h \, ds,
$$
and therefore $f = h$ on $\Gamma$ and $0$ on $\Gamma_c$. We have obtain that $f \in H^1(\Omega)$ verifies $-\nabla \cdot \sigma^T \nabla f = 0$
in $\Omega$ and $f = \sigma^T \nabla f \cdot \nu = 0$ on $\Gamma_c$: by uniqueness of the solution of the elliptic data completion problem,
necessarily $f = 0$, which implies directly $\mathbf{F} = \mathbf{0}$ and $g = h = 0$.

Finally, let us prove that $A$ is not a compact operator. Consider $e_n$  an Hilbert basis of $L^2(\Omega)$, and $u_n$ 
in $H^1(\Omega)$ verifying $\displaystyle \int_\Omega u_n\, dx = 0$, $ -\nabla \cdot \sigma \nabla u_n = e_n$ in $\Omega$ and $\sigma \nabla u_n \cdot \nu = 0$ on $\partial \Omega$. It is not difficult to show that $u_n$ exists and is unique. Furthermore, $\Vert u_n \Vert_{H^1(\Omega)} \leq C(\Omega, \sigma)$, and in particular $\Vert u_n \Vert_{L^2(\Gamma)} \leq C(\Omega,\sigma)$. Defining $\mathbf{p}_n := \sigma \nabla u_n \in \tilde{H}_\mathsf{div}(\Omega)$, we obtain $\Vert \mathbf{p}_n \Vert_{\tilde{H}_\mathsf{div}} \leq C(\Omega,\sigma)$, hence
 $(u_n,\mathbf{p}_n)$ is a bounded sequence in $\mathcal{X}$. But $A(u_n,\mathbf{p_n}) = (\mathbf{0},e_n,u_{n\, \vert \Gamma},0)$ does not admit any convergent
 subsequence.
\end{proof}

To define our quasi-reversibility approach to this data completion problem, we  choose  
$b(.,.)$  such that the corresponding norm $\Vert . \Vert_{A,b}$ is equivalent
to the norm $\Vert . \Vert_\mathcal{X}$. Of course, we could choose the whole $\mathcal{X}$-scalar product. But it might be interesting 
to use another form, to soften the regularization: here we define 
$$
b\Big( (u,\mathbf{p}) , (v,\mathbf{q} ) \Big) := \int_\Omega ( \nabla u \cdot \nabla v + \mathbf{p}\cdot \mathbf{q})\, dx
$$
which  is a symmetric bilinear non-negative form in $\mathcal{X}$ (but obviously not a scalar product). Using Poincar\'e inequality, it is easy to obtain the existence of $c,C>0$ such that
$$
c \Vert v,\mathbf{q} \Vert_\mathcal{X} \leq \Vert v,\mathbf{q} \Vert_{A,b} \leq C \Vert v,\mathbf{q} \Vert_X.
$$
As for any $(v,\mathbf{q}) \in \mathcal{X}$, $\Vert v,\mathbf{q} \Vert_{b} \leq \Vert v,\mathbf{q} \Vert_\mathcal{X}$, for a fixed $\ep$
the regularization term in the quasi-reversibility method is smaller.

Applying the abstract setting to this problem, we obtain the following quasi-reversibility regularization: for $\ep>0$, find 
$(u_\ep, \mathbf{p}_\ep) \in H^1(\Omega) \times \tilde{H}_{\mathsf{div}(\Omega)}$ such that for all $(v,\mathbf{q}) \in H^1(\Omega)
\times \tilde{H}_{\mathsf{div}}(\Omega)$, 
$$
\int_\Omega (\sigma \nabla u_\ep - \mathbf{p}_\ep) \cdot (\sigma \nabla v - \mathbf{q})\, dx
+ \int_\Omega (\nabla \cdot \mathbf{p}_\ep)\, (\nabla \cdot \mathbf{q})\, dx +
\int_\Gamma \Big( u_\ep\, v + (\mathbf{p}_\ep \cdot \nu) \, (\mathbf{q}\cdot \nu) \Big)\, ds
$$
$$
+ \ep\int_\Omega ( \nabla u_\ep \cdot \nabla v + \mathbf{p}_\ep\cdot \mathbf{q})\, dx =
 -\int_\Omega f\, (\nabla \cdot \mathbf{q}) dx  + \int_\Gamma 
 \Big( g_D\, v + g_N\, (\mathbf{q}\cdot \nu) \Big)\, ds.
$$
This problem always admits an unique solution $(u_\ep, \mathbf{p}_\ep)$. We know from our study that if the data completion problem
for the Poisson's equation admits a solution $(u_s,\mathbf{p}_s)$, then $(u_\ep,\mathbf{p}_\ep)$ converges
monotonically  to $(u_s,\mathbf{p}_s)$ as $\ep$ goes to zero, with the estimate
$$
\sqrt{\Vert \sigma \nabla u_\ep - p_\ep \Vert_{L^2(\Omega)^d}^2 + \Vert \nabla \cdot p_\ep - f \Vert_{L^2(\Omega)}^2 + 
\Vert u_\ep - g_D \Vert_{L^2(\Gamma)}^2 + \Vert \mathbf{p}_\ep \cdot \nu - g_N \Vert_{L^2(\Gamma)}^2 }$$
$$
\leq \sqrt{\ep} \Vert u_s,\mathbf{p}_s \Vert_b.
$$
If not, we know that $\Vert (u_\ep, \mathbf{p}_\ep ) \Vert_b \xrightarrow[\ep \rightarrow 0]{} +\infty$.

The quasi-reversibility method we obtain in this study is close to the one proposed in \cite{Dar2} to stabilize the data completion
problem, which was:  find 
$(u_\ep, \mathbf{p}_\ep) \in H^1(\Omega) \times \tilde{H}_{\mathsf{div}(\Omega)}$, $u_\ep = g_D$
and $\mathbf{p}_\ep\cdot \nu = g_N$ on $\Gamma$, such that for all $(v,\mathbf{q}) \in H^1(\Omega)
\times \tilde{H}_{\mathsf{div}}(\Omega)$, $v = \mathbf{q}\cdot \nu = 0$ on $\Gamma$,
\begin{align*}
\int_\Omega (\sigma \nabla u_\ep - \mathbf{p}_\ep) \cdot & (\sigma \nabla v - \mathbf{q})\, dx
 + \int_\Omega (\nabla \cdot \mathbf{p}_\ep)\, (\nabla \cdot \mathbf{q})\, dx \\
& + \ep\ \Big( (u_\ep,\mathbf{p}_\ep), (v,\mathbf{q}) \Big)_\mathcal{X} =
 -\int_\Omega f\, (\nabla \cdot \mathbf{q}) dx.
\end{align*}
The two differences are first the use of $b(.,.)$ instead of $(.,.)_\mathcal{X}$ in the regularization term,
and secondly in the way the boundary condition are included in the problem. In the formulation proposed 
in \cite{Dar2}, they are strongly imposed, which presents two main issues: one is theoretical, as the regularized problem
might not have solution if $g_D$ is in $L^2(\Gamma)$, and not in $H^{1/2}(\Gamma)$, as in that case
there is no $v \in H^1(\Omega)$ such that $v = g_D$ on $\Gamma$, and therefore $u_\ep$ cannot exist.
The second one is practical: it is not a good idea to strongly impose data which might extremely noisy, as
in that case the noise is somehow imposed to the solution.
In the quasi-reversibility regularization obtain in the present paper, the boundary conditions are weakly imposed, which solves
both of the problems: the regularized problem always admits a solution, even in the case where $g_D$ is not the trace on $\Gamma$
of a $H^1$ function, and the noise is regularized directly by the formulation, leading to a stabler formulation.

Finally, the abstract iterated quasi-reversibility method applied to the elliptic data completion problem is: for $\ep> 0$, $(u_\ep^{-1}, \mathbf{p}_\ep^{-1}) = (0,\mathbf{0})$ and for all $M \in \mathbb{N}$, $(u_\ep^M,\mathbf{^p}_\ep^M) \in H^1(\Omega) \times \tilde{H}_{\mathsf{div}}(\Omega)$ verifies
$$
\int_\Omega (\sigma \nabla u_\ep^M - \mathbf{p}_\ep^M) \cdot (\sigma \nabla v - \mathbf{q})\, dx
+ \int_\Omega (\nabla \cdot \mathbf{p}_\ep^M)\, (\nabla \cdot \mathbf{q})\, dx 
$$
$$
+
\int_\Gamma \Big( u_\ep^M\, v + (\mathbf{p}_\ep^M \cdot \nu) \, (\mathbf{q}\cdot \nu) \Big)\, ds
+ \ep\, b\Big( (u_\ep^M,\mathbf{p}_\ep^M), (v,\mathbf{q}) \Big) =$$
$$
 -\int_\Omega f\, (\nabla \cdot \mathbf{q}) dx  + \int_\Gamma 
 \Big( g_D\, v + g_N\, (\mathbf{q}\cdot \nu) \Big)\, ds
 + \ep\, b\Big( (U_\ep^{M-1},\mathbf{P}_\ep^{M-1}), (v,\mathbf{q}) \Big).
$$
and we directly know that $u_\ep^M$ and $\mathbf{p}_\ep^M$ converge to  $u$ and $\sigma \nabla u$ when $M$ goes to infinity. 
In the case where noisy data $f^\delta$, $g_D^\delta$ and $g_N^\delta$ are available, such that 
$$
\Vert g_D - g_D^\delta \Vert_{L^2(\Gamma)}^2 + \Vert g_N - g_N^\delta \Vert_{L^2(\Gamma)}^2 +\Vert f- f^\delta \Vert_{L^2(\Omega)}^2
\leq \delta^2,
$$
in accordance with the result of section \ref{section_iterated_noisy_data}, 
we stop the iterations the first time that $\Vert A (u_\ep^M, \mathbf{p}_\ep^M) - (\mathbf{0},f^\delta, g_D^\delta , g_N^\delta) \Vert_\mathcal{Y} \leq r\delta $, with $r>1$ close to 1.

\begin{remark}
Actually, in the following numerical results, we use $r=1$.
\end{remark}  

\subsection{Heat equation}

As for the Poisson problem, we modify the data completion problem defined in the introduction, introducing the flux $p := \nabla u$
as an additional unknown:

\begin{problem}
find $(u,\mathbf{p})$ in $L^2(0,T;H^1(\Omega)) \cap H^1(0,T; L^2(\Omega)) \times L^2(0,T; \tilde{H}_{\mathsf{div}}(\Omega))$
such that
$$
\left\lbrace
\begin{array}{rcl}
\partial_t u - \nabla \cdot \mathbf{p} &=& f \text{ in } \mathcal{Q} \\
\nabla u & = & \mathbf{p} \text{ in } \mathcal{Q} \\
u & = & g_D \text{ on } (0,T) \times \Gamma \\
\mathbf{p} \cdot \nu & = & g_N  \text{ on } (0,T) \times \Gamma
\end{array}
\right.
$$
\end{problem}

Again, we define 
\begin{align*}
A : &  (u,\mathbf{p}) \in\mathcal{X} := L^2(0,T;H^1(\Omega)) \cap H^1(0,T; L^2(\Omega)) \times L^2(0,T; \tilde{H}_{\mathsf{div}}(\Omega)) \\
  & \mapsto (\nabla u - \mathbf{p}, \partial_t u -\nabla \cdot \mathbf{p}, u_{\vert \Gamma}, \mathbf{p}\cdot \nu_{\vert \Gamma})  \in
\mathcal{Y}=L^2(0,T; L^2(\Omega)^d) \times L^2(0,T;L^2(\Omega)) \\ & \times L^2(0,T;L^2(\Gamma)) \times L^2(0,T;L^2(\Gamma)).
\end{align*}

Here, the spaces $\mathcal{X}$ and $\mathcal{Y}$ are endowed with their natural scalar products, respectively
\begin{align*}
 \Big( (u,\mathbf{p}) ,  (v,\mathbf{q})  \Big)_\mathcal{X} :=   \int_0^T \int_\Omega 
 \Big( \partial_t u \, \partial_t v + \nabla u \cdot \nabla v + u\, v + & (\nabla \cdot \mathbf{p}) \, (\nabla \cdot \mathbf{q}) + \mathbf{p}
 \cdot \mathbf{q} \Big) dx\, dt\\
+ &\int_0^T \int_\Gamma (\mathbf{p}\cdot \nu ) \, (\mathbf{q}\cdot \nu) \, dS \, dt
\end{align*}
and 
\begin{align*}
\Big( (\mathbf{F}_1 ,f_1 , g_1, h_1),&  (\mathbf{F}_2 ,f_2 , g_2, h_2)   \Big)_\mathcal{Y}  :=
\int_0^T \int_\Omega \Big( \mathbf{F}_1 \cdot \mathbf{F}_2  + f_1\, f_2 \Big) dx\, dt \\
 & + \int_0^T \int_\Gamma \Big( g_1 \, g_2  + h_1\, h_2 \Big) dS\, dt.
\end{align*}
It is then not difficult to verify the 

\begin{proposition}
$A$ is a linear continuous. It is one-to-one but not onto, and has dense image. Furthermore, it is not a compact operator.
\end{proposition}

\begin{proof}
We will just prove prove that $A$ has dense range, as it is not difficult to be convinced that $A$ is non compact, and the rest of the proposition follows directly  from the definition of $A$, $\mathcal{X}$ and $\mathcal{Y}$, and the ill-posedness of the corresponding data completion problem.  

Let $\mathbf{F} \in L^2(0,T; L^2(\Omega)^d )$, $f \in L^2(0,T;L^2(\Omega))$, $g\in L^2(0,T;L^2(\Gamma))$ and $h \in L^2(0,T;L^2(\Gamma))$
be such that for all $v \in L^2(0,T;H^1(\Omega)) \cap H^1(0,T;L^2(\Omega))$ and all $\mathbf{q} \in L^2(0,T;\tilde{H}_\textsf{div}(\Omega))$,
$\Big(A(v,\mathbf{q}) , (\mathbf{F},f,g,h) \Big)_\mathcal{Y} =0$, that is
$$
\int_0^T \int_\Omega \Big( (\nabla v - \mathbf{q}) \cdot \mathbf{F} + (\partial_t v - \nabla \cdot \mathbf{q})\, f\Big)  dx
+ \int_0^T \int_\Gamma \Big( v\, g + (\mathbf{q}\cdot \nu) h \Big) ds = 0.
$$
First of all, choosing $\mathbf{q} = \Upsilon \in C^\infty_c(\Omega)^d$, we obtain $\mathbf{F} = \nabla f$, and therefore
$f \in L^2(0,T;H^1(\Omega))$. So we have
\begin{equation} \label{eqn_density_A}
\int_0^T \int_\Omega \Big( (\nabla v - \mathbf{q}) \cdot \nabla f+ (\partial_t v - \nabla \cdot \mathbf{q})\, f\Big)  dx
+ \int_0^T \int_\Gamma \Big( v\, g + (\mathbf{q}\cdot \nu) h \Big) ds = 0.
\end{equation}
For all $\mathbf{q} \in L^2(0,T;\tilde{H}_\textsf{div}(\Omega))$, for almost all $t \in (0,T)$, we have
$$
\int_\Omega \Big( \mathbf{q} \cdot \nabla f +(\nabla \cdot \mathbf{q}) f \Big) dx = \langle\mathbf{q}\cdot \nu, f \rangle_{H^{-1/2}(\partial\Omega),H^{1/2}(\partial \Omega)}
$$
which leads by integration in time to
$$
\int_0^T \langle\mathbf{q}\cdot \nu, f \rangle_{H^{-1/2}(\partial\Omega),H^{1/2}(\partial \Omega)} dt = 
- \int_0^T \int_\Gamma (\mathbf{q}\cdot \nu) h ds\, dt
$$
that is $f = -h$ on $(0,T)\times \Gamma$ and $f = 0$ on $(0,T)\times \Gamma_c$.

Now, taking $v= \varphi \in C^\infty_c(\Omega)$ in (\ref{eqn_density_A}) leads to $\partial_t f + \Delta f = 0$ in $(0,T)\times \Omega$.  We see that $\mathbf{V} := (f,\nabla f) \in H_{\mathsf{div}}(\mathcal{Q})$, and we can apply the divergence theorem:
for all $v \in H^{1,1}(\mathcal{Q}) = L^2(0,T;H^1(\Omega)) \cap H^1(0,T;L^2(\Omega))$, we have
$$
\int_\mathcal{Q}  \Big( (\nabla_{t,x} \cdot \mathbf{V} ) v + \mathbf{V} \cdot \nabla_{t,x} v \Big) dx\,dt
 = \langle \mathbf{V}\cdot \nu, v \rangle_{H^{-1/2}(\partial \mathcal{Q}) , H^{1/2}(\partial \mathcal{Q})}.
$$
that is, taking any $v \in H^{1,1}(\mathcal{Q}) $  such that $v(0,x) = v(T,x) = 0$,
$$
\int_0^T \int_\Omega \Big( \underset{=0}{\underbrace{(\partial_t f + \Delta f)}} v + f\, \partial_t v + \nabla f \cdot \nabla v \Big) dx dt 
= \langle \langle \nabla f\cdot \nu, v \rangle_{H^{-1/2}(\partial \Omega), H^{1/2}(\partial \Omega)} \rangle_{H^{-1/2}(0,T), H^{1/2}(0,T)}
$$
leading to $\nabla f \cdot \nu = -g$ on $(0,T)\times\Gamma$ and $\nabla f \cdot \nu = 0$ on $(0,T)\times\Gamma_c$.
Therefore, $f$ verifies $\partial_t f + \Delta f =0$ in $(0,T) \times \Omega$, $f = \nabla f \cdot \nu = 0$ on $(0,T)\times \Gamma_c$,
hence $f \equiv 0$ in $(0,T) \times \Gamma$, leading to $\mathbf{F} = \mathbf{0}$ and $g = h = 0$.
\end{proof}

Similarly as for the previous regularization, we introduce the symmetric bilinear non-negative form
$$
b\Big( (u,\mathbf{p}) , (v,\mathbf{q}) \Big) := \int_{0}^T \int_\Omega \big( \partial_t u \, \partial_t v + \nabla u \cdot \nabla v + \mathbf{p} \cdot \mathbf{q} \Big) dx\, dt.
$$
It is easy to check that the bilinear form $(A (u,\mathbf{p}), A (v,\mathbf{q}) )_\mathcal{Y} + b( (u,\mathbf{p}), (v,\mathbf{q})) $
is a scalar product on $\mathcal{X}$, and that it exists two constants $c, C>0$ such that
$$
c \Vert u,\mathbf{p} \Vert_\mathcal{X} \leq \Vert  u,\mathbf{p}  \Vert_{A,b} \leq C \Vert  u,\mathbf{p}  \Vert_\mathcal{X}.
$$
The quasi-reversibility regularization we consider is therefore: for $\ep>0$, find $(u_\ep,\mathbf{p}_\ep) \in \mathcal{X}$
such that for all $(v,\mathbf{q}) \in \mathcal{X}$, we have
$$
\int_0^T \int_\Omega \Big( (\partial_t u_\ep - \nabla \cdot \mathbf{p}_\ep) (\partial_t v  - \nabla \cdot \mathbf{q}) +
(\nabla u_\ep - \mathbf{p}_\ep) \cdot (\nabla v - \mathbf{q}) \Big) dx\, dt 
$$
$$
+ \int_0^T \int_\Gamma \Big( u_\ep \, v + (\mathbf{p}_\ep \cdot \nu) (\mathbf{q}\cdot \nu) \Big) dS \, dt 
+ \ep  \int_{0}^T \int_\Omega \big( \partial_t u_\ep \, \partial_t v + \nabla u_\ep \cdot \nabla v + \mathbf{p}_\ep \cdot \mathbf{q} \Big) dx\, dt
$$
$$
= \int_0^T \int_\Omega f\,(\partial_t v  - \nabla \cdot \mathbf{q})\, dx\, dt  +  \int_0^T \int_\Gamma \Big(g_D \, v + g_N (\mathbf{q}\cdot \nu) \Big) dS \, dt.
$$
According to our present study,
this problem always admits a unique solution $(u_\ep,\mathbf{p}_\ep) $ that converges to $(u, \nabla u)$ when $\ep$ goes to zero.
The corresponding iterated quasi-reversibility is: $(u_\ep^{-1},\mathbf{p}_\ep^{-1}) = (0,\mathbf{0}) $ and for all 
$M \in \mathbb{N}$, $(u_\ep^{M},\mathbf{p}_\ep^M)$ is such that for all $(v,\mathbf{q})$,
$$
\int_0^T \int_\Omega \Big( (\partial_t u_\ep^M - \nabla \cdot \mathbf{p}_\ep^M) (\partial_t v  - \nabla \cdot \mathbf{q}) +
(\nabla u_\ep^M - \mathbf{p}_\ep^M) \cdot (\nabla v - \mathbf{q}) \Big) dx\, dt 
$$
$$
+ \int_0^T \int_\Gamma \Big( u_\ep^M \, v + (\mathbf{p}_\ep^M \cdot \nu) (\mathbf{q}\cdot \nu) \Big) dS \, dt 
+ \ep  \int_{0}^T \int_\Omega \big( \partial_t u_\ep^M \, \partial_t v + \nabla u_\ep^M \cdot \nabla v + \mathbf{p}_\ep^M \cdot \mathbf{q} \Big) dx\, dt
$$
$$
= \int_0^T \int_\Omega f\,(\partial_t v  - \nabla \cdot \mathbf{q})\, dx\, dt  +  \int_0^T \int_\Gamma \Big(g_D \, v + g_N (\mathbf{q}\cdot \nu) \Big) dS \, dt
$$
$$
+ \ep  \int_{0}^T \int_\Omega \big( \partial_t u_\ep^{M-1} \, \partial_t v + \nabla u_\ep^{M-1} \cdot \nabla v + \mathbf{p}_\ep^{M-1} \cdot \mathbf{q} \Big) dx\, dt.
$$

\section{Numerical results}

\subsection{Elliptic equation} \label{sect_num_ell}

We consider a domain $\Omega \subset \mathbb{R}^2$, with exterior boundary $\Gamma$ and interior boundary $\Gamma_c$ defined by $$\partial \Gamma := 
 \left\lbrace  r(t) \begin{bmatrix}
 \cos(t) \\ \sin(t)
 \end{bmatrix}\ t \in [0,2\,\pi]\right\rbrace \quad \partial \Gamma_c := 
 \left\lbrace  r_{c}(t) \begin{bmatrix}
 \cos(t) \\ \sin(t)
 \end{bmatrix}\ t \in [0,2\,\pi]\right\rbrace,$$
 with
 $$
 r(t) = 1 + 0.1\cos(2\,t) - 0.05\sin(3\,t),\quad r_{c}(t) = 0.5 -0.02\cos(t) + 0.1\sin(t).
 $$
We consider the problem of reconstructing a Robin coefficient  $\eta$ on $\partial \Gamma_c$ from the knowledge of 
a  noisy Cauchy data $(g_D^\delta,g_N^\delta)$ on $\partial \Gamma$. Mathematically, we want to reconstruct a function
$u \in H^1(\Omega)$ and a function $\eta \in C^2(\Gamma_c)$ such that 
$$
\left\lbrace 
\begin{array}{rcll}
-\Delta u &  = & 0 & \text{ in } \Omega \\
u & = & g_D^\delta & \text{ on }  \Gamma \\
\partial_\nu u & = & g_N^\delta & \text{ on } \Gamma \\
\partial_\nu u + \eta\, u & = & 0  & \text{ on } \Gamma_c
\end{array}
\right.
$$
The Cauchy data $(g_D^\delta,g_N^\delta)\in L^2(\Gamma) \times L^2(\Gamma)$ is supposed to correspond to an exact data $(g_D,g_N)$ corrupted by some noise of amplitude $\delta$ : 
$$\Vert g_D^\delta - g_D \Vert_{L^2(\Gamma)}^2 + \Vert g_N^\delta - g_N \Vert_{L^2(\Gamma)}^2 \leq \delta^2.$$
Our strategy to reconstruct $\eta$ is therefore to compute $u_\ep^{M(\delta)}$ and $\mathbf{p}_\ep^{M(\delta)}$, approximations of $u$ and $\nabla u$ with the prescribed noisy Cauchy data
on $\Gamma$ and no data at all on $\Gamma_c$, using the iterated quasi-reversibility method for the Poisson problem. Then, we obtain
an approximation $\eta_\ep$ of $\eta$ on $\Gamma_c$ by simply taking the ratio $\displaystyle \eta_\ep =  - \frac{\mathbf{p}_\ep^{m(\delta)} \cdot \nu}{u_\ep^{M(\delta)}}$.

In our experiments, $\eta = 0.5+0.3\, \sin(2\,(\theta-5\pi/4))$, $\theta$ being the polar angle of a point $\mathbf{x}$ on $\Gamma_c$, and
$g_N = 1$. The corresponding Dirichlet data is obtained by solving the direct problem $-\Delta u = 0$ in $\Omega$,
$\partial_\nu u = 0.2$ on $\Gamma$ and $\partial_\nu u + \eta u  =  0  $ on $ \Gamma_c$ using a finite-element method, and defining 
$g_D := u_{\vert \Gamma}$. 

\begin{figure}[htb]

\includegraphics[width=0.49\textwidth]{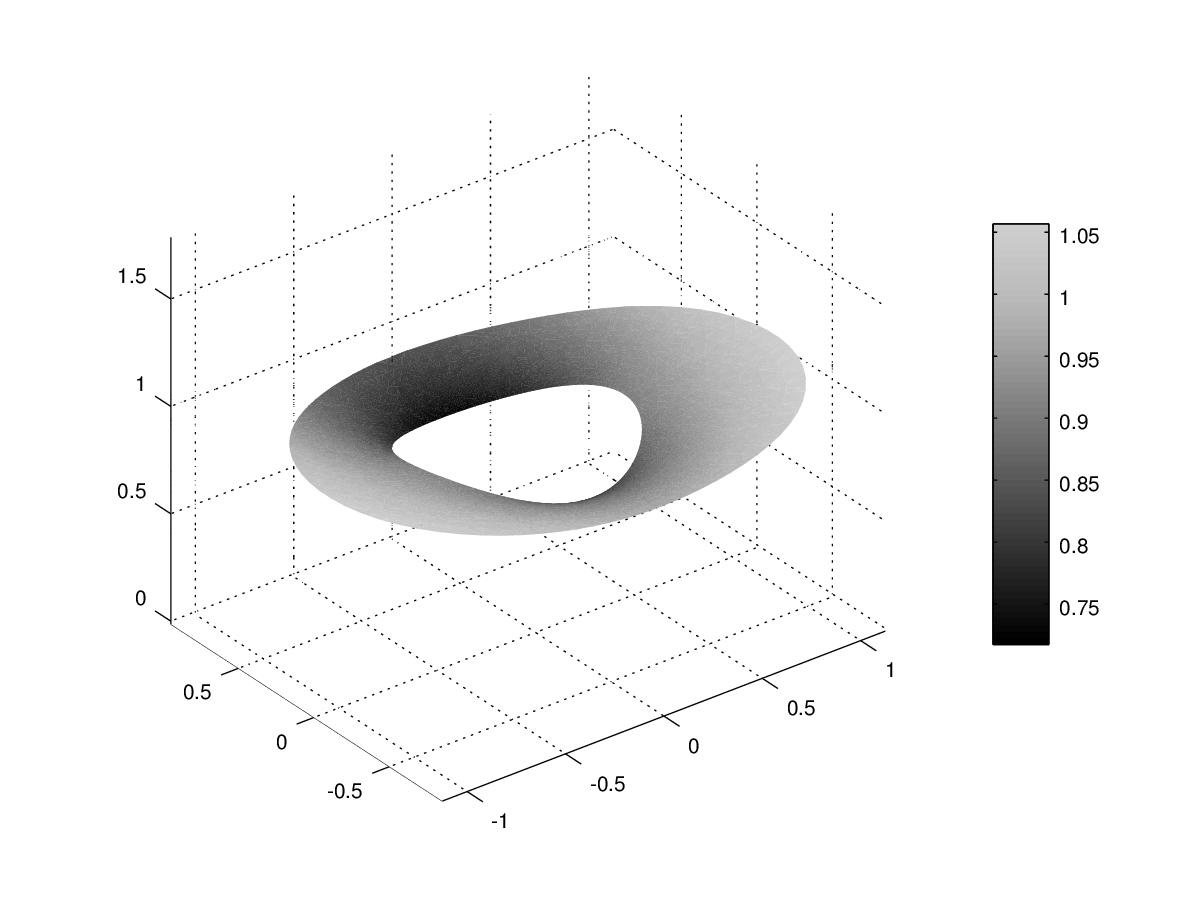}
\includegraphics[width=0.49\textwidth]{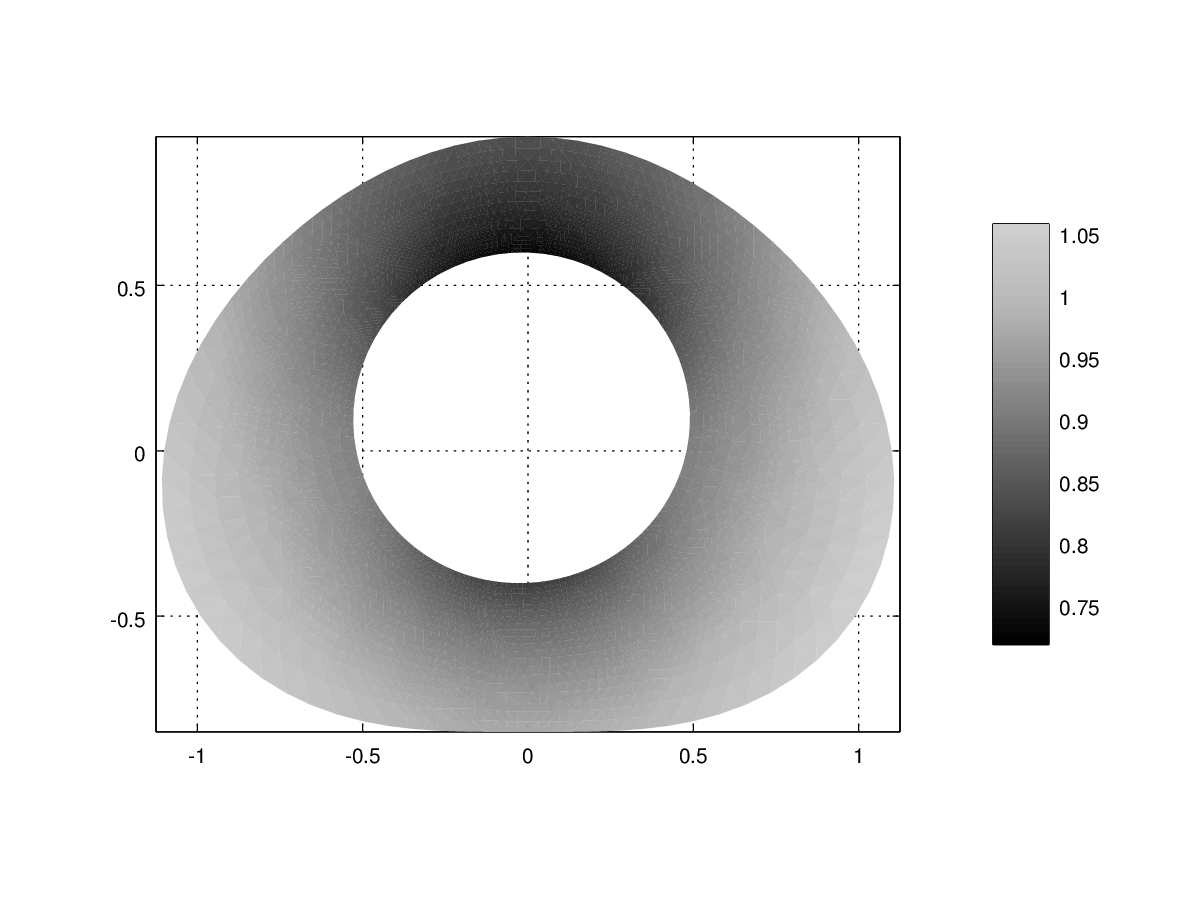}

\caption{Exact solution.}
\end{figure}

Then we corrupt the Dirichlet data $g_D$ pointwise with a normal noise having zero mean and variance one, to obtained the corrupted Dirichlet data
$g_D^\delta$. The noise is scaled so that
$$
\Vert g_D^\delta - g_D \Vert_\infty = \alpha \Vert g_D \Vert_\infty
$$
that is the relative amplitude of noise in $L^\infty$-norm is $\alpha$. 
In the experiments, we have chosen $\alpha = 1\%$, $2\%$ and $5\%$.
The exact Neumann data is used (\textit{i.e.} $g_N^\delta = g_N$), as
in practical situations it is the imposed data (the net current), whereas the $g_D$ is the measured data (the corresponding voltages). Therefore
$g_N$ is known quite precisely compared to $g_D$. We then compute the corresponding amplitude of noise for the $L^2$ norm
$\delta = \delta(\alpha, \Vert g_D \Vert_\infty)$, which defined our stopping criterion for the iteration of the method.

The iterated quasi-reversibility problem is then solved using a conforming finite-element method using $P_2$ Lagrange finite elements
for $u_\ep^M$ and $RT_1$ Raviart-Thomas finite elements for $\mathbf{p}_\ep^M$ \cite{Ciar}. 
The study of convergence of the finite-element approximation of the quasi-reversibility approximation toward the continuous solution is
just a slight adaptation of section 4.4 in \cite{Dar2}, as the formulations are quite similar, and therefore is omitted in the present study.
To avoid an inverse crime, the direct and inverse problems are solved on different meshes.

\begin{figure}[htb] 
\includegraphics[width=0.49\textwidth]{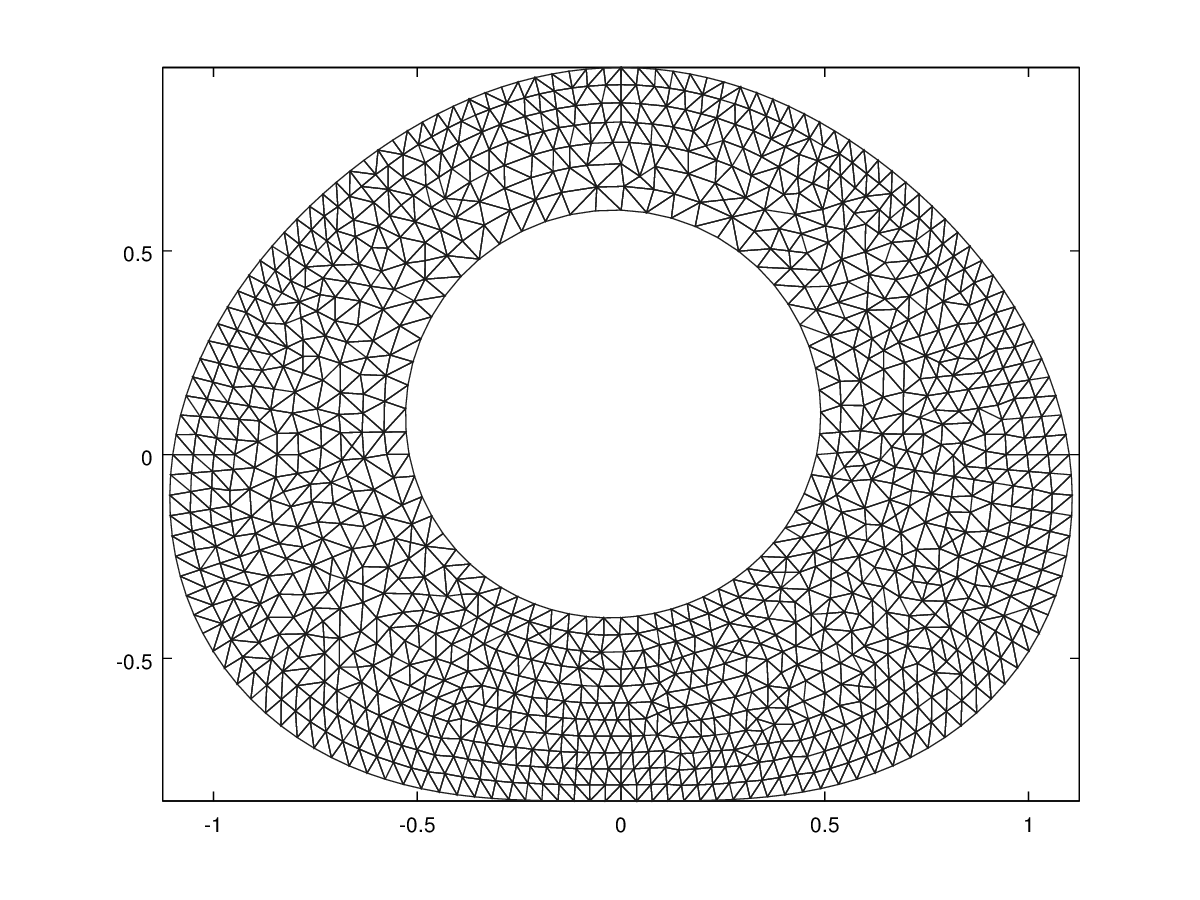}
\includegraphics[width=0.49\textwidth]{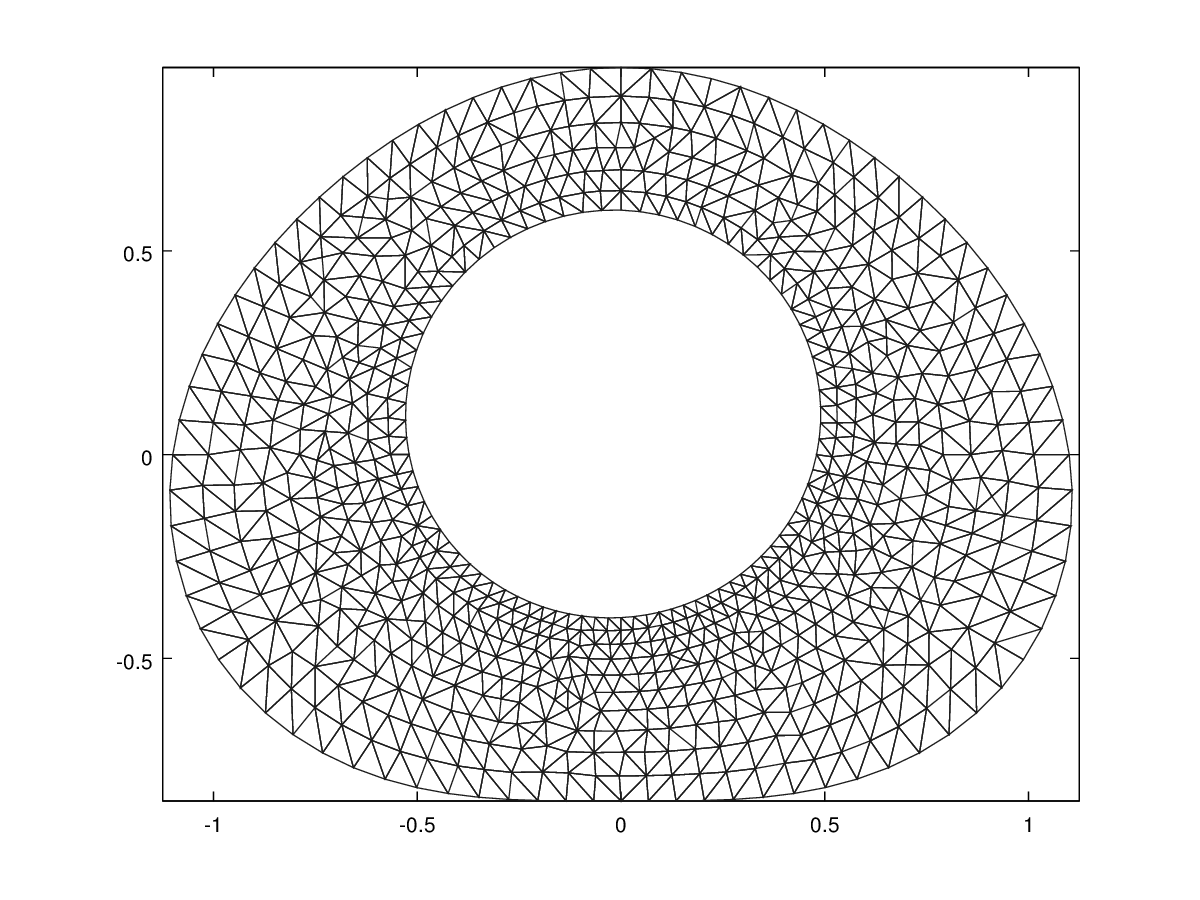}

\caption{The direct and inverse meshes used in the numerical simulation.}\label{fig_meshes}
\end{figure}

According to our study, the choice of $\varepsilon$ is completely arbitrary in the iterated quasi-reversibility method. Therefore,
we have chosen $\ep = 1$ in the experiments, as it leads to a good conditioning of the finite-element matrices.

First of all, we present in figure \ref{figure_res} the evolution of the residual 
$$
\sqrt{\Vert \nabla \cdot \mathbf{p}_\ep^M \Vert_{L^2(\Omega)}^2 + \Vert \nabla u_\ep^M - \mathbf{p}_\ep^M \Vert_{L^2(\Omega)}^2
+ \Vert u_\ep - g_D^\delta \Vert_{L^2(\Gamma)}^2 + \Vert \mathbf{p}_\ep \cdot \nu - g_N^\delta \Vert_{L^2(\Gamma)}  }
$$
until the stopping criterion is reached. As expected theoretically, the greater is the noise, the smaller is $M(\delta)$.

\begin{figure}[htp] 
\includegraphics[width=0.49\textwidth]{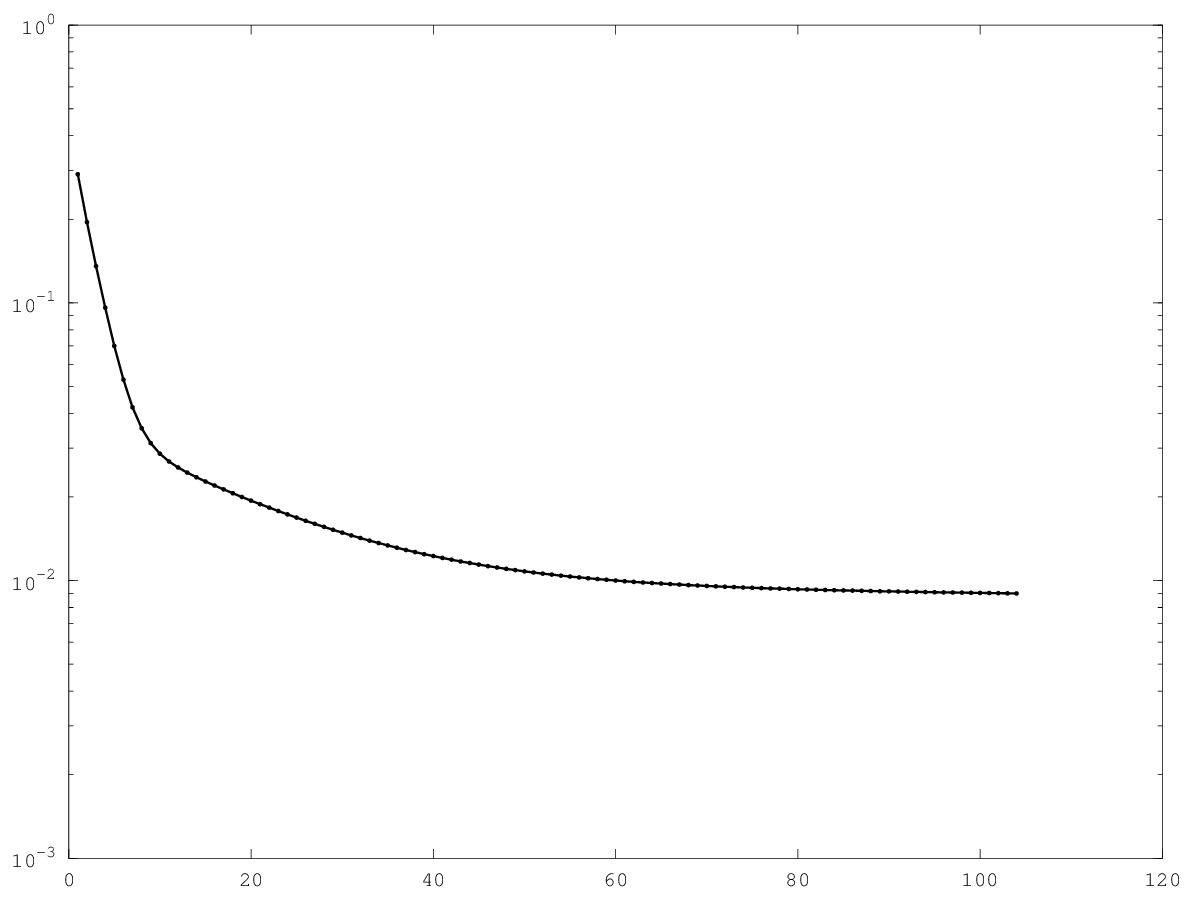}
\includegraphics[width=0.49\textwidth]{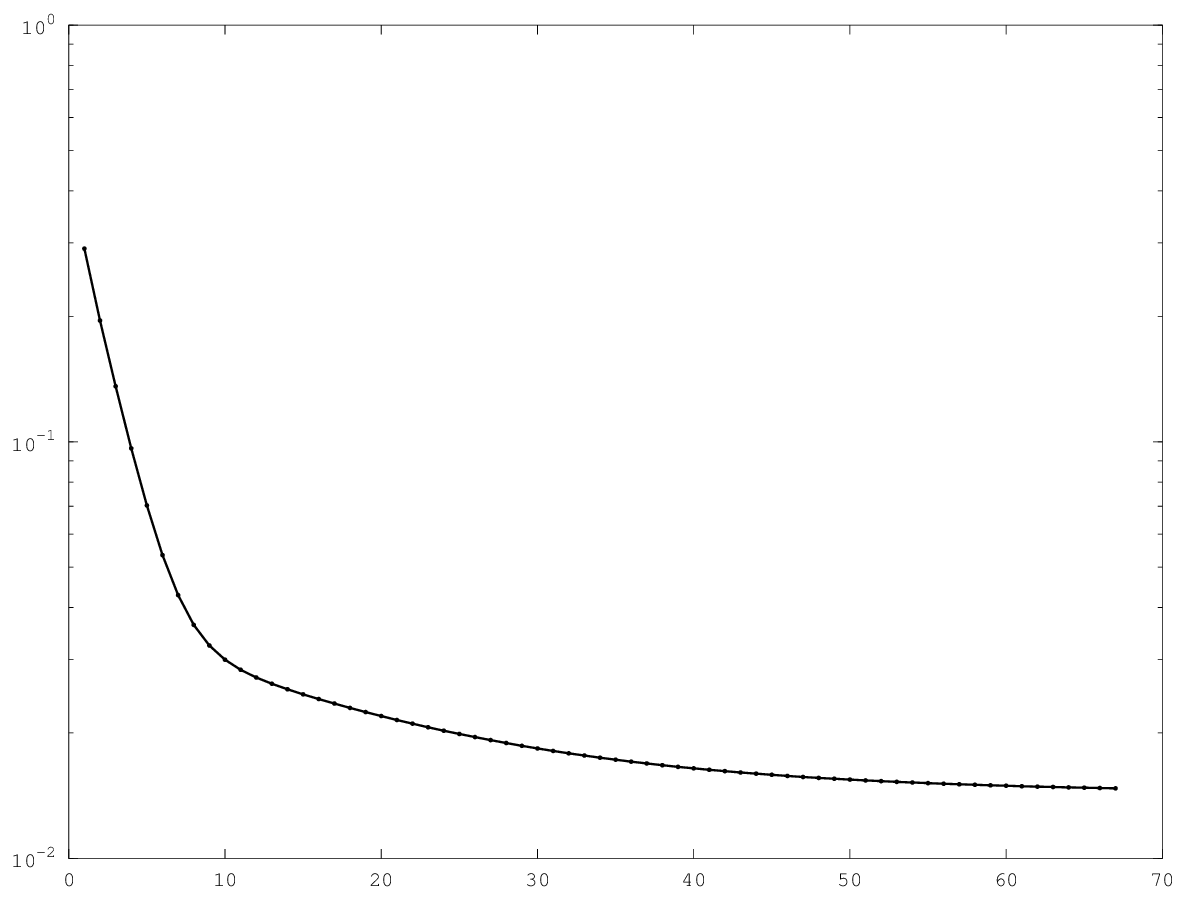}
\includegraphics[width=0.49\textwidth]{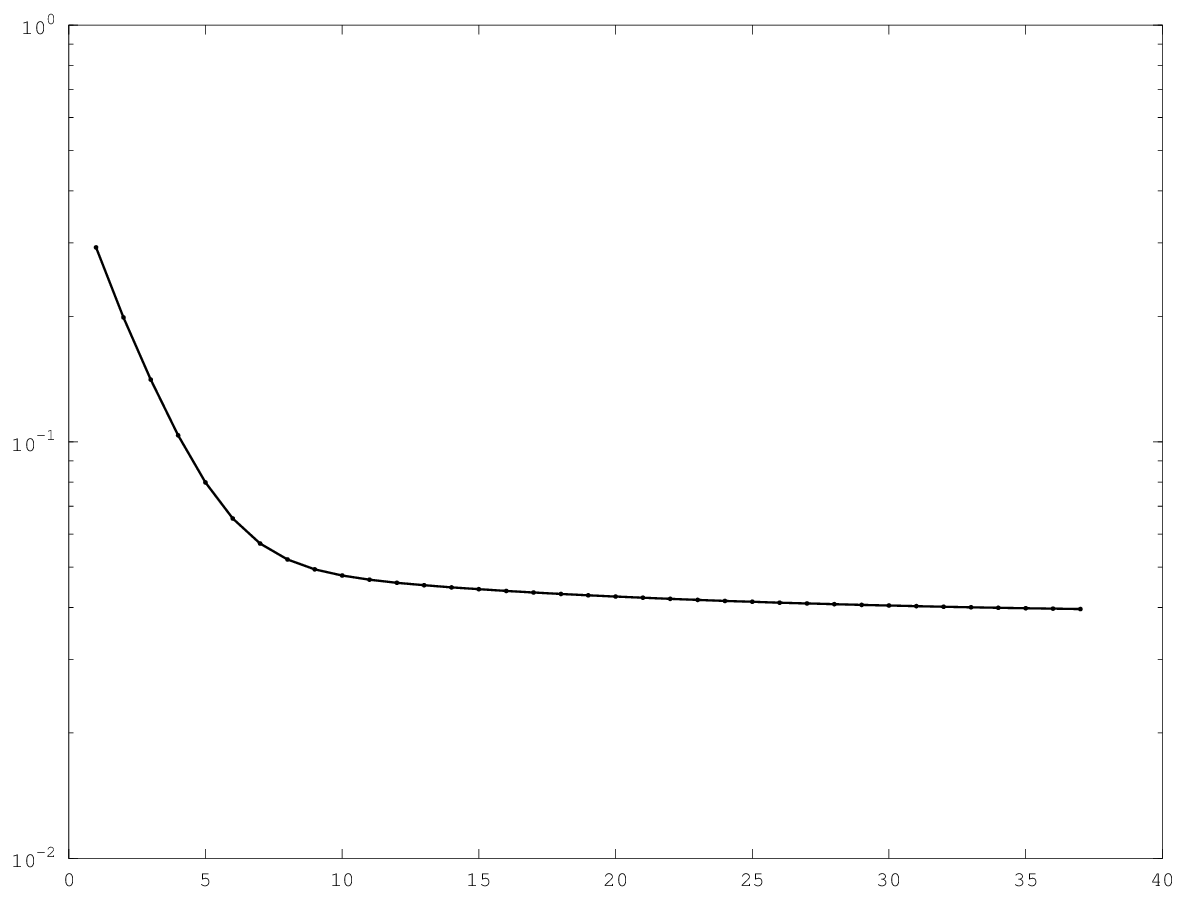}

\caption{Residual vs. number of iterations, until stopping criterion is reached, for $\alpha$ = $1\%$, $2\%$ and $5\%$.}\label{figure_res}
\end{figure}

Now we present the reconstruction results: in figure \ref{figure_uepminu}, the exact solution is compared to the reconstructed one in the whole
domain of study. In figure \ref{figure_bnddata}, we focus on the boundary $\Gamma$: we compare the exact data $g_D$, the noisy one $g_D^\delta$
used in the iterated quasi-reversibility method, and finally the trace of the reconstructed solution $u_\ep^{M(\delta)}$. Note that the iterated quasi-reversibility method gives good result even with severely corrupted data.

\begin{figure}[htp] 
\includegraphics[width=0.49\textwidth]{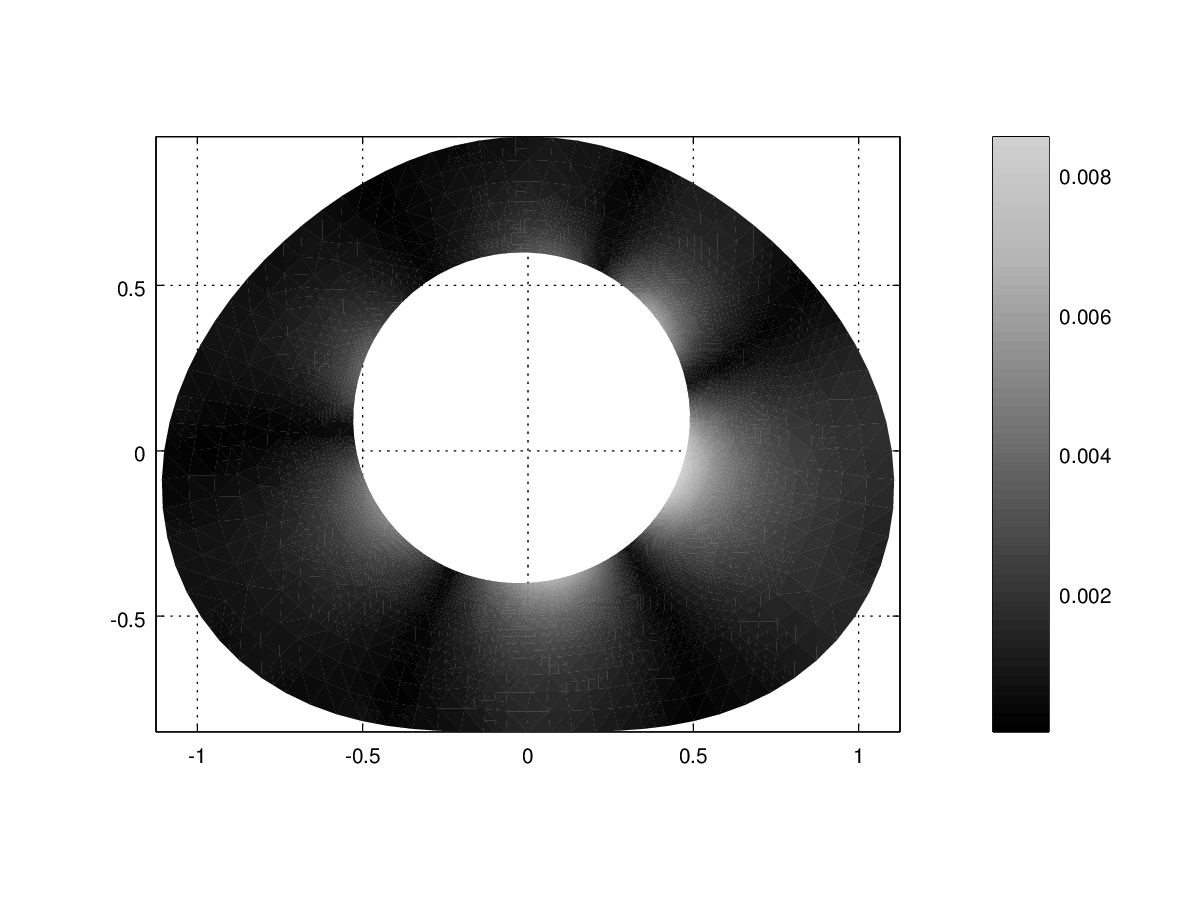}
\includegraphics[width=0.49\textwidth]{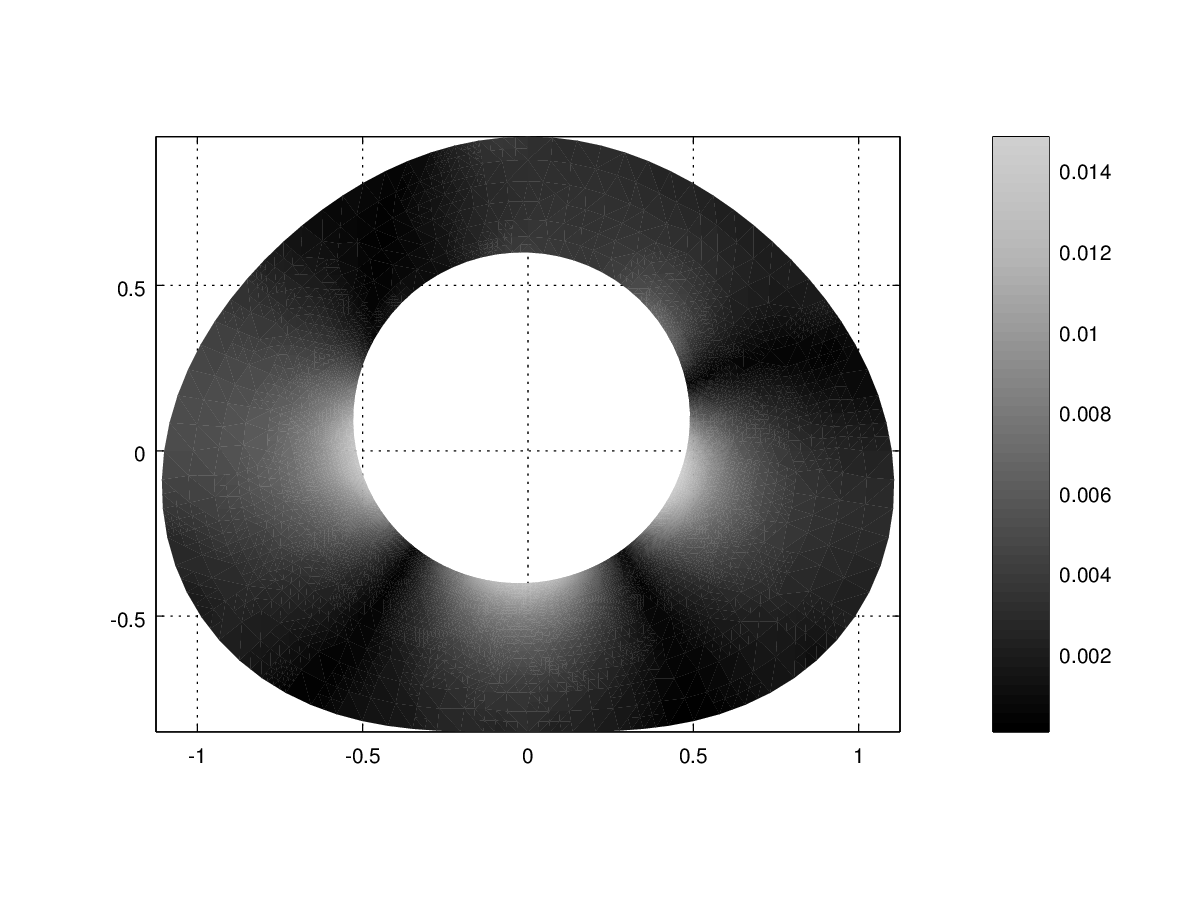}
\includegraphics[width=0.49\textwidth]{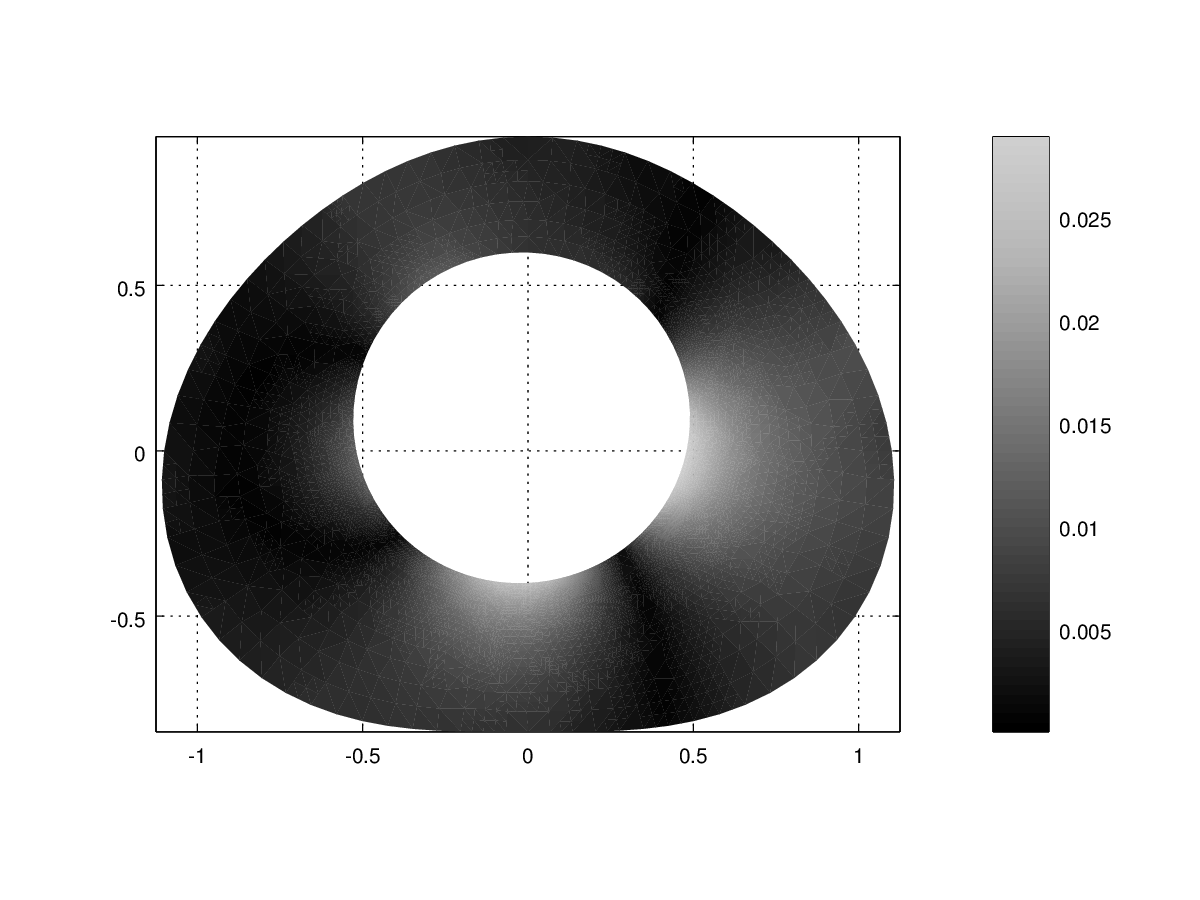}

\caption{$\vert u_\ep^{M(\delta)} - u \vert$, for $\alpha$ = $1\%$, $2\%$ and $5\%$.}\label{figure_uepminu}
\end{figure}

\begin{figure}[htp] 
\includegraphics[width=0.49\textwidth]{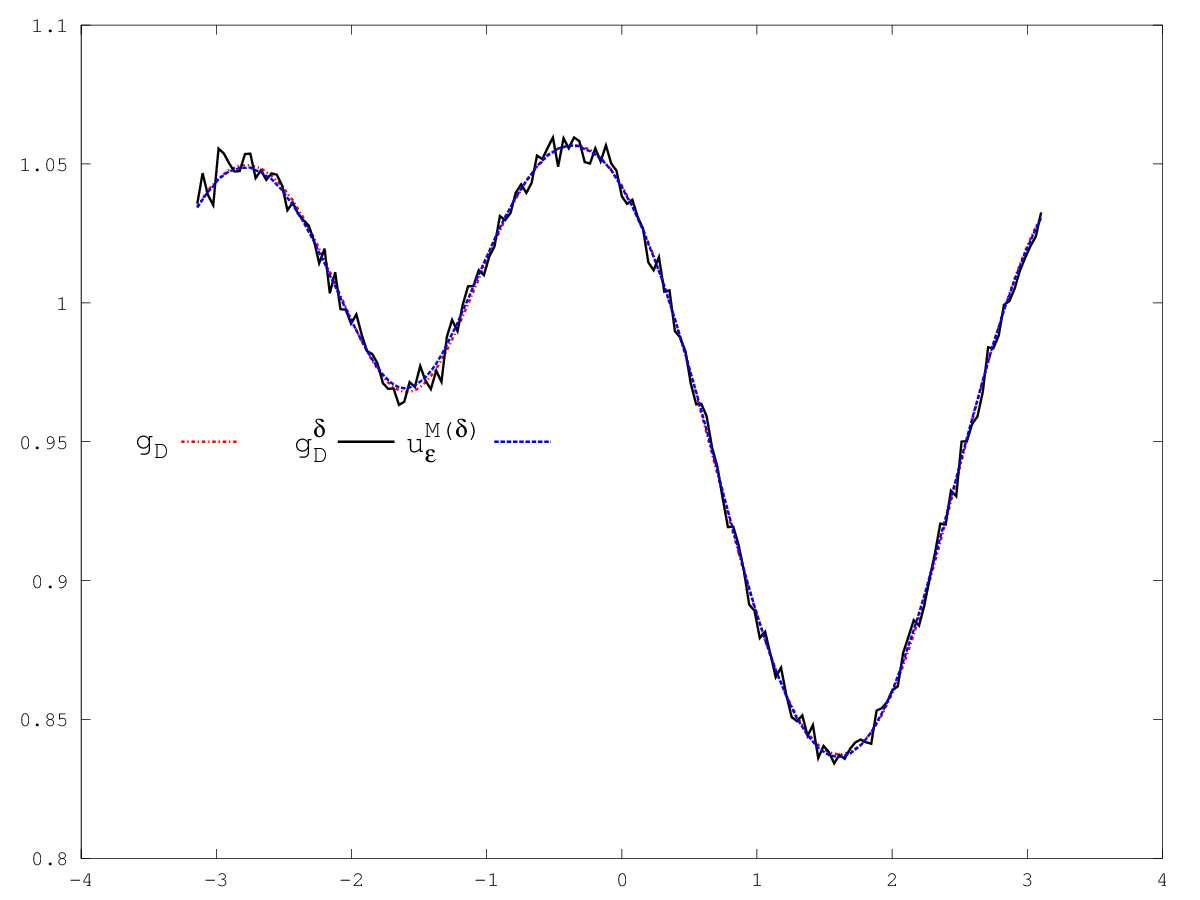}
\includegraphics[width=0.49\textwidth]{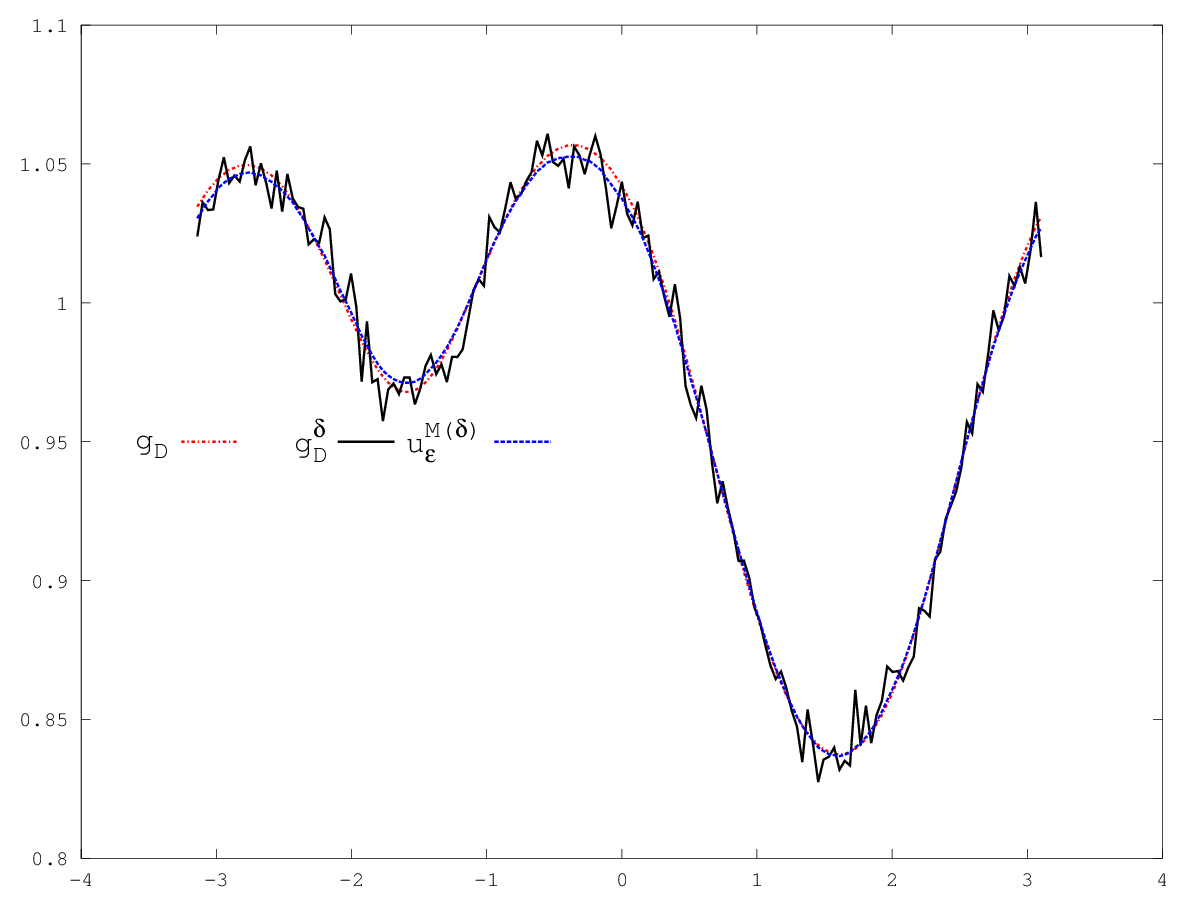}
\includegraphics[width=0.49\textwidth]{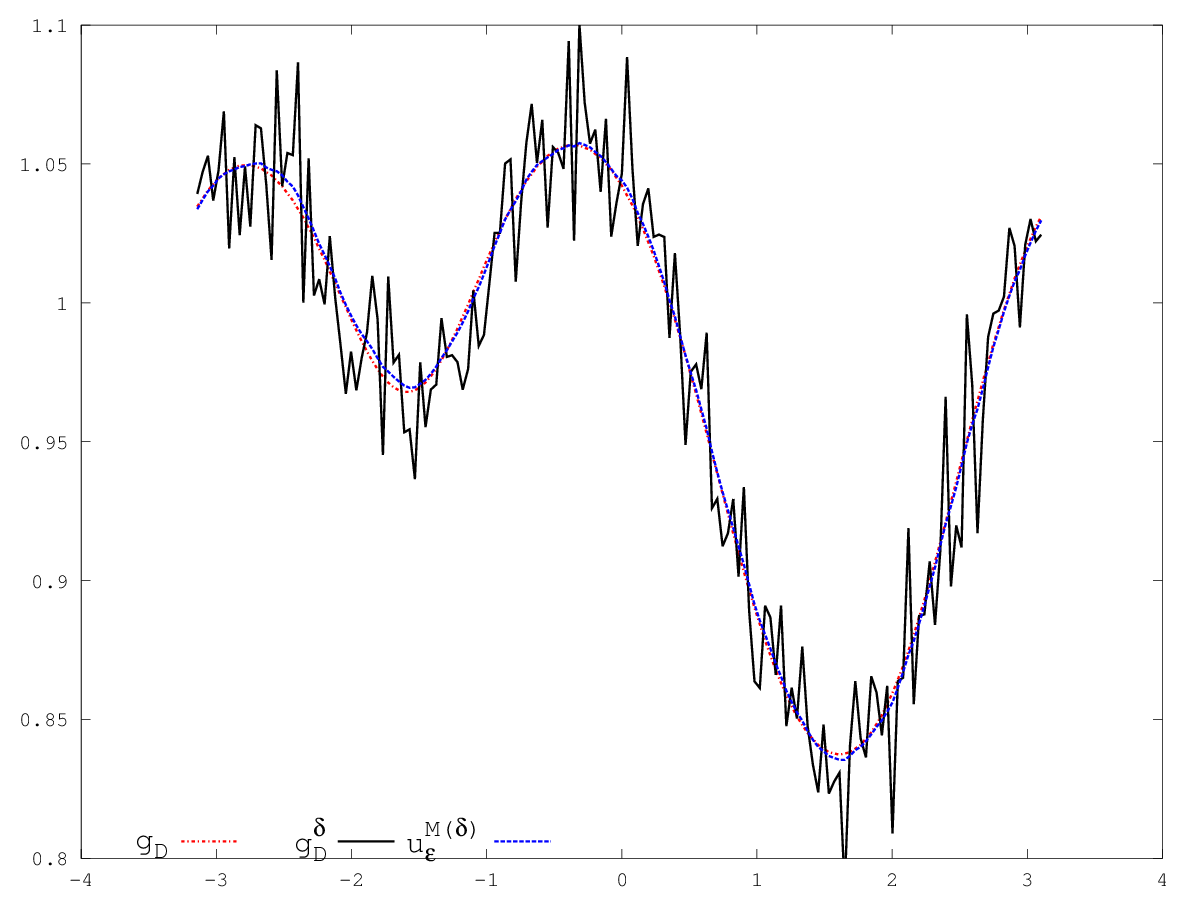}

\caption{Exact Dirichlet data, noisy Dirichlet data, and trace of $u_\ep^{M(\delta)}$ on $\Gamma$, for $\alpha$ = $1\%$, $2\%$ and $5\%$.}\label{figure_bnddata}
\end{figure} 

Finally, on figure \ref{figure_Robinrecons}, we present the reconstructed Robin coefficient on $\Gamma_c$, which was our main objective. Again, the
reconstruction is still acceptable for high level of noise on the data.

\begin{figure}[htp] 
\includegraphics[width=0.49\textwidth]{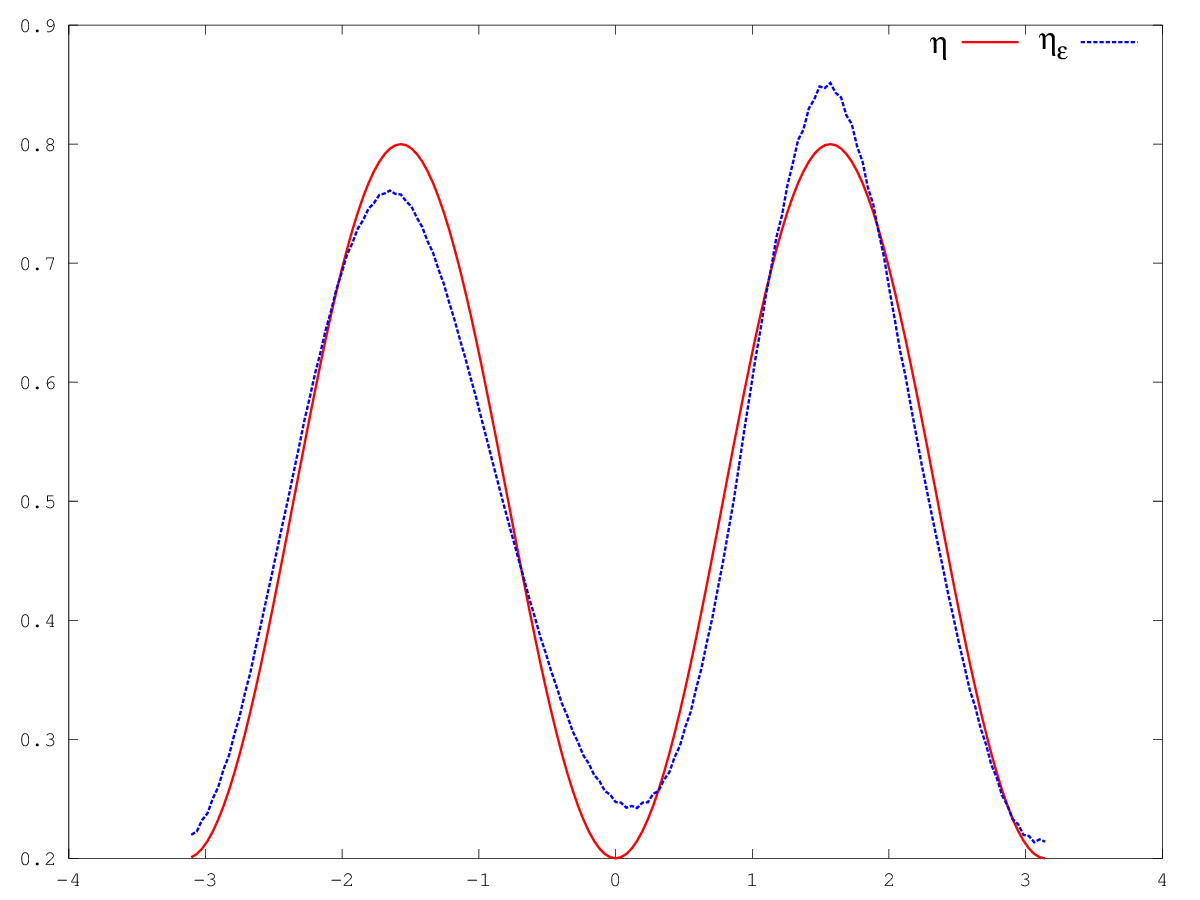}
\includegraphics[width=0.49\textwidth]{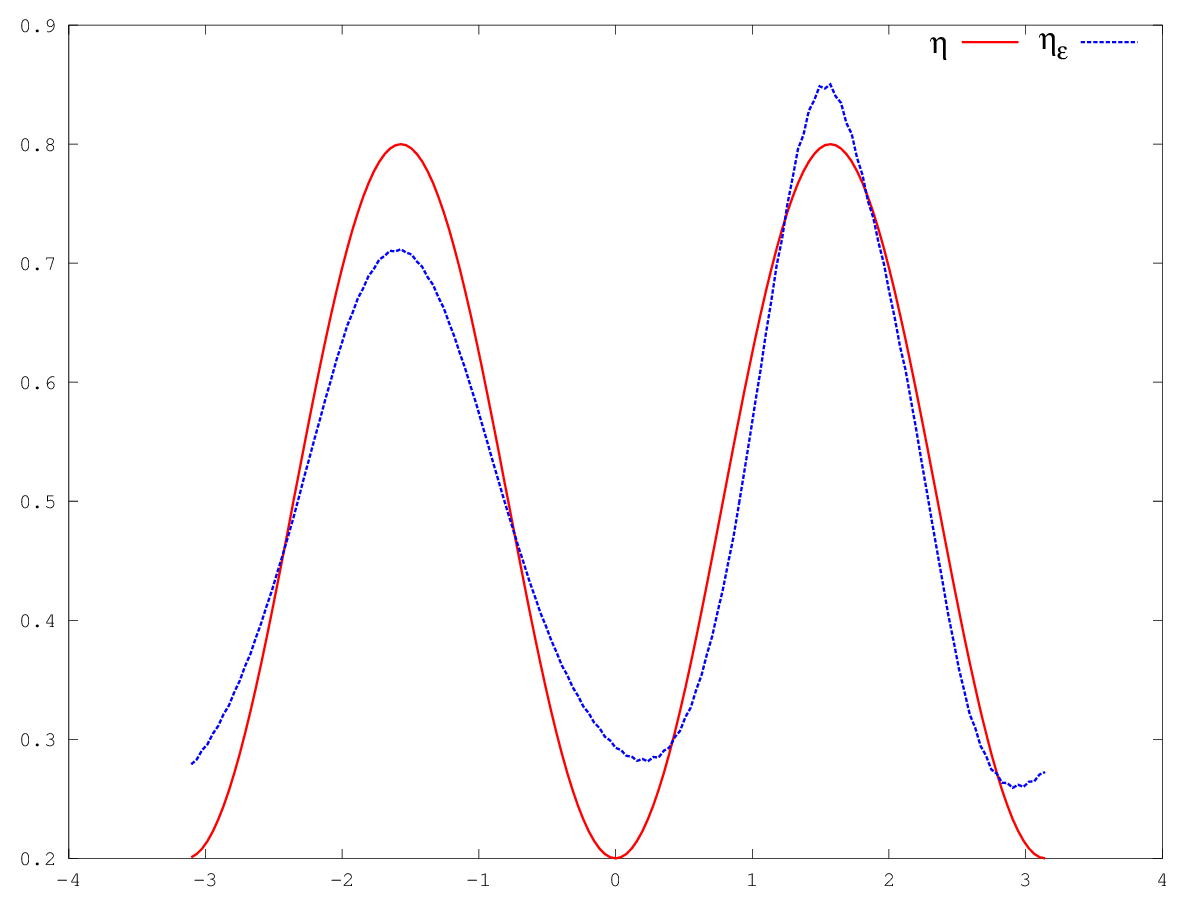}
\includegraphics[width=0.49\textwidth]{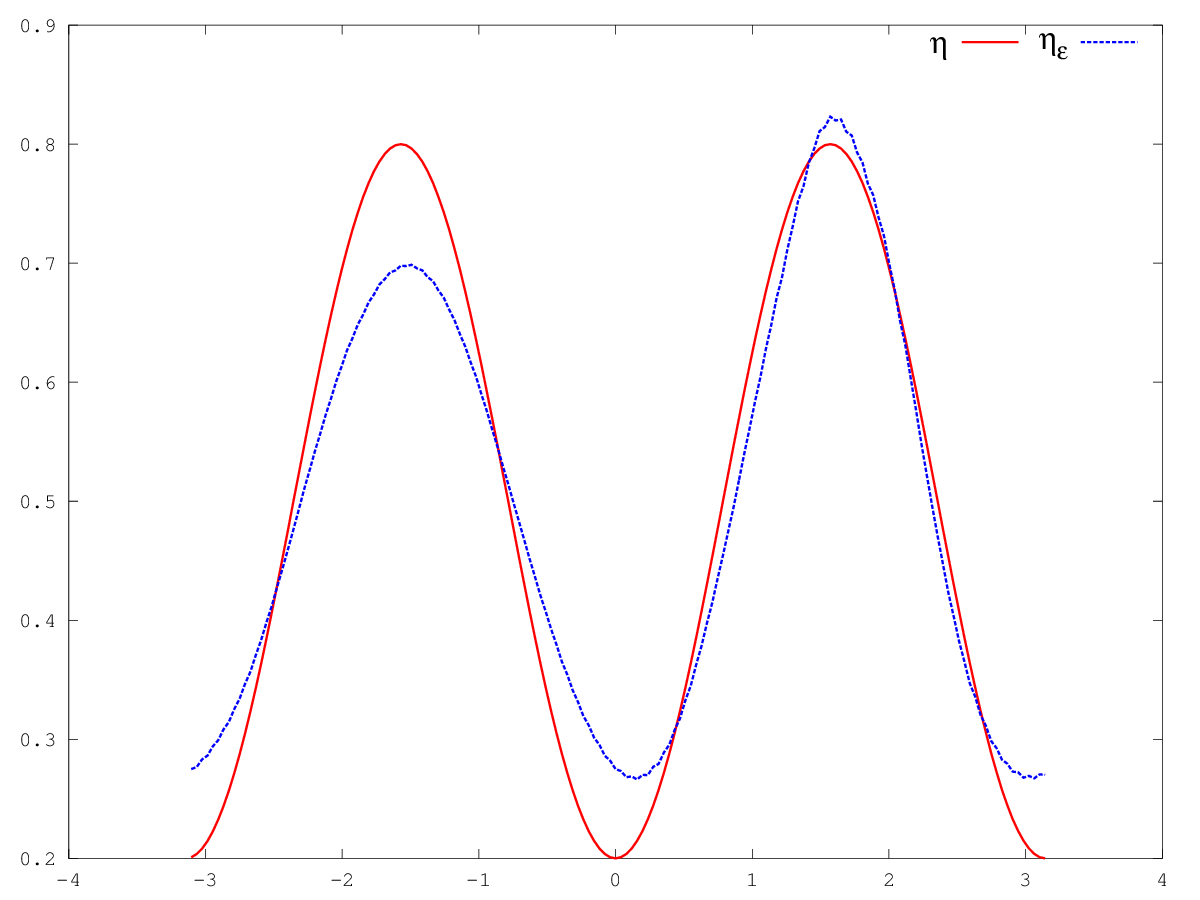}

\caption{Exact ($\eta$) and reconstructed ($\eta_\ep$) Robin coefficient on $\Gamma_c$, for $\alpha$ = $1\%$, $2\%$ and $5\%$.}\label{figure_Robinrecons}
\end{figure}

\subsection{One-dimensional heat equation}

We now focus on the data-completion problem for a one-dimensional heat equation. The problem reads: find $u \in H^{1,1}((0,T)\times (a,b))$
such that
$$
\left\lbrace 
\begin{array}{cccl}
\partial_t u   &= & \partial_{xx} u  & \text{in } (0,T)\times (a,b)\\
u(t,a) & = & g_D^\delta ,  & \ t \in (0,T) \\
\partial_x u(t,a) & = & g_N^\delta, & \ t \in (0,T).
\end{array}
\right.
$$
Note that, as $\partial_{xx}u = \partial_t u \in L^2(0,T;L^2(a,b))$, we have $p := \partial_x u \in L^2(0,T;H^1(a,b))$, and hence
$g_N(t) = p(a,t) \in L^2(0,T)$ without additional assumption, which is not the case for the multi-dimensional case.
Hence the equivalent data-completion problem with additional unknown $p$ 
reads: for $(g_D^\delta,g_N^\delta) \in L^2(0,T)\times L^2(0,T)$, find $u \in H^{1,1}((0,T)\times(a,b))$ and $p \in L^2(0,T;H^1(a;b))$ such that
$$
\left\lbrace 
\begin{array}{cccl}
\partial_t u   &= & \partial_x p  & \text{in } (0,T)\times (a,b)\\
\partial_x u  & = &  p & \text{in } (0,T) \times (a,b)\\
u(t,a) & = & g_D^\delta ,  & \ t \in (0,T) \\
p(t,a) & = & g_N^\delta, & \ t \in (0,T).
\end{array}
\right.
$$
According to our study, the quasi-reversibility regularization of this problem is: for $\ep>0$, find $u_\ep \in H^{1,1}((0,T)\times(a,b))$
and $p_\ep \in L^2(0,T;H^1(a,b))$ such that for all $v \in H^{1,1}((0,T)\times (a,b))$, for all $q \in L^2(0,T;H^1(a,b))$
$$
\int_0^T \int_a^b \Big( ( \partial_t u_\ep - \partial_x p_\ep)\, (\partial_t v - \partial_x q)  + (\partial_x u_\ep - p_\ep) \, (\partial_x v - q )\Big) \, dx\, dt
$$
$$
+ \int_0^T  \Big( u_\ep(s,a) \, v(s,a)  + p_\ep(a,s) \, q(a,s) \Big) \, dt +
\ep \int_0^T \int_a^b \Big( \partial_t u_\ep \, \partial_t v + \partial_x u_\ep \, \partial_x v + p_\ep \, q \Big) \, dx \, dt
$$
$$
= \int_0^T \Big( g_D^\delta(s) \, v(s,a)  + g_N^\delta(s) \, q(s,a)\Big) \, ds,
$$
and the iterated quasi-reversibility method: for $\ep >0$, define $(u_\ep^{-1}, p_\ep^{-1} )  = (0,0)$ and for all $M \in \mathbb{N}$,
$u_\ep^{M} \in H^{1,1}((0,T)\times(a,b))$
and $p_\ep^{M} \in L^2(0,T;H^1(a,b))$ are such that for all $v \in H^{1,1}((0,T)\times (a,b))$, for all $q \in L^2(0,T;H^1(a,b))$,
$$
\int_0^T \int_a^b \Big( ( \partial_t u_\ep^M - \partial_x p_\ep^M)\, (\partial_t v - \partial_x q)  + (\partial_x u_\ep^M - p_\ep^M) \, (\partial_x v - q )\Big) \, dx\, dt
$$
$$
+ \int_0^T  \Big( u_\ep^M(s,a) \, v(s,a)  + p_\ep^M(a,s) \, q(a,s) \Big) \, dt +
\ep \int_0^T \int_a^b \Big( \partial_t u_\ep^M \, \partial_t v + \partial_x u_\ep^M \, \partial_x v + p_\ep^M \, q \Big) \, dx \, dt
$$
$$
= \int_0^T \Big( g_D^\delta(s) \, v(s,a)  + g_N^\delta(s) \, q(s,a)\Big) \, ds + \ep \int_0^T \int_a^b \Big( \partial_t u_\ep^{M-1} \, \partial_t v + \partial_x u_\ep^{M-1} \, \partial_x v + p_\ep^{M-1} \, q \Big) \, dx \, dt.
$$
We discretize the space $H^{1,1}(\mathcal{Q})$ and $L^2(0,T;H^1(a,b))$ using a tensorial product of Lagrange finite elements, namely
$P^1 \otimes P^1$ finite elements for $H^{1,1}$ and $P^0 \times P^1$ for $L^2(H^1)$.

In our simulations, we choose $T = 1$, $a=1$ and $b =2$.
We consider two exact solution of the heat equation $\displaystyle u_1(t,x) := \frac{1}{8} \left( \frac{x^3}{3} + x\, (1+2\,t)\right)$
and $\displaystyle u_2(t,x) := e^{-t/4} \sin\left(t/2\right)$.

\begin{figure}[htb]

\includegraphics[width=0.49\textwidth]{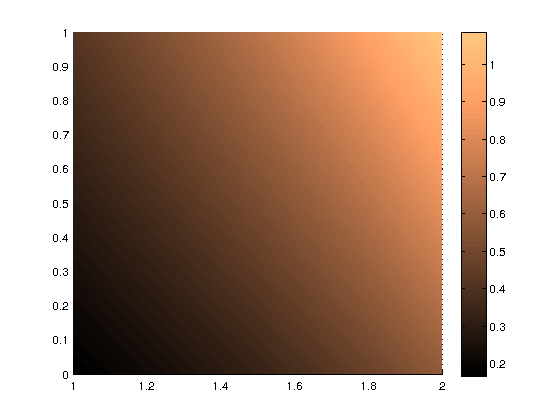}
\includegraphics[width=0.49\textwidth]{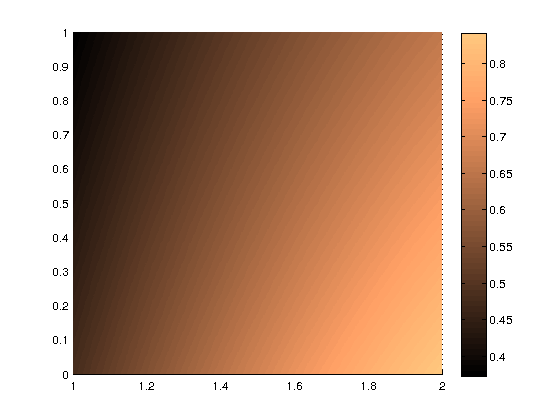}

\caption{Exact solutions $u_1$ and $u_2$ in $\mathcal{Q}$.}

\end{figure}

The corresponding exact data $(g_D,g_N)$ are corrupted pointwise by a normal noise with zero means and variance one, which is scaled so that the 
noisy data $(g_D^\delta, g_N^\delta)$ verifies
$$
\Vert g_D^\delta - g_D \Vert_\infty = \alpha \Vert g_D \Vert_\infty,\quad \Vert g_N^\delta - g_N \Vert_\infty = \alpha \Vert g_N \Vert_\infty.
$$
In our experiments, we test our method with $\alpha = 2\%$ and $\alpha = 5\%$.  As in the elliptic case, we choose $\ep = 1$, and 
stop the iterations of the method once the stopping criterion is reached.

In figures \ref{fig_relerrpol} and \ref{fig_relerrex}, we present the  relative error over $\mathcal{Q}$, defined as the ratio
$$
\frac{u_\ep^M(\delta) - u}{\Vert u \Vert_\infty}
$$
for both solutions $u_1$ and $u_2$. We see that the iterated quasi-reversibility method gives also good reconstruction for this parabolic problem,
even for high level of noise on both Dirichlet and Neumann data.

\begin{figure}[htb]

\includegraphics[width=0.49\textwidth]{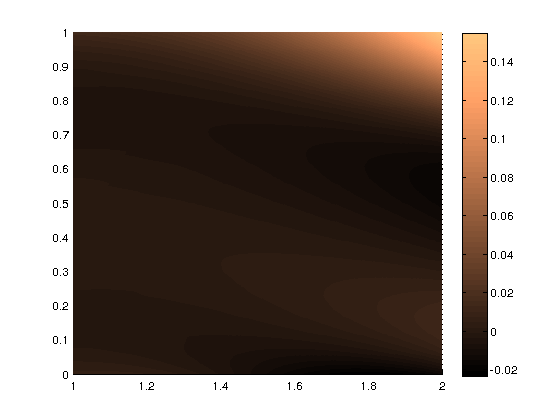}
\includegraphics[width=0.49\textwidth]{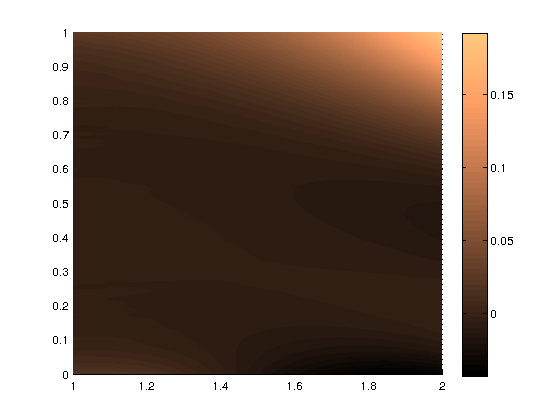}

\caption{Relative error $(u_{1 \ep}^{M(\delta)} - u_1) / \Vert u_1 \Vert_\infty$ in $\mathcal{Q}$. Left: $\alpha = 2\%$. Right: $\alpha = 5\%$.}\label{fig_relerrpol}

\end{figure}

\begin{figure}[htb]

\includegraphics[width=0.49\textwidth]{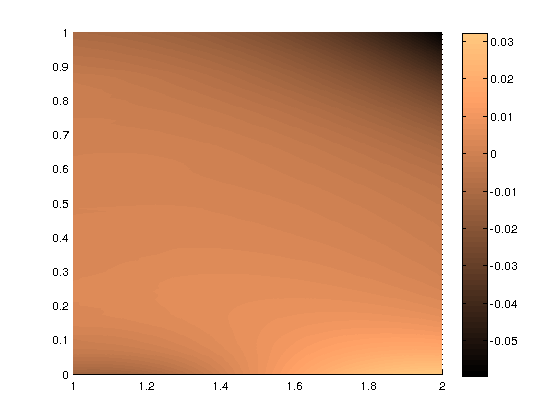}
\includegraphics[width=0.49\textwidth]{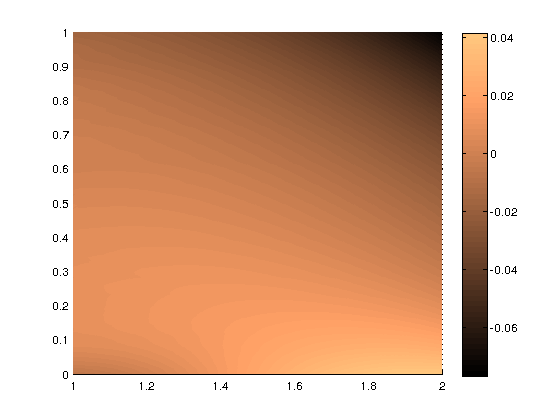}

\caption{Relative error $(u_{2 \ep}^{M(\delta)} - u_2) / \Vert u_2 \Vert_\infty$ in $\mathcal{Q}$. Left: $\alpha = 2\%$. Right: $\alpha = 5\%$.}\label{fig_relerrex}

\end{figure}

Finally, in figures \ref{fig_respol}, we present the evolution of the residual quantity
{\small
$$
\sqrt{ \Vert \partial_t u_\ep^M - \partial_x p_\ep^M \Vert_{L^2(\mathcal{Q})}^2 + \Vert \partial_x u_\ep^M - p_\ep^M\Vert_{L^2(\mathcal{Q})}^2 
 +  \Vert u_\ep^M(.,a) - g_D^\delta \Vert_{L^2(0,T)}^2 + \Vert p_\ep^M(.,a) - g_N^\delta \Vert_{L^2(0,T)}^2 }
$$ }
during the iterations of the method, until the stopping criterion is reached.

\begin{figure}[htb]

\includegraphics[width=0.49\textwidth]{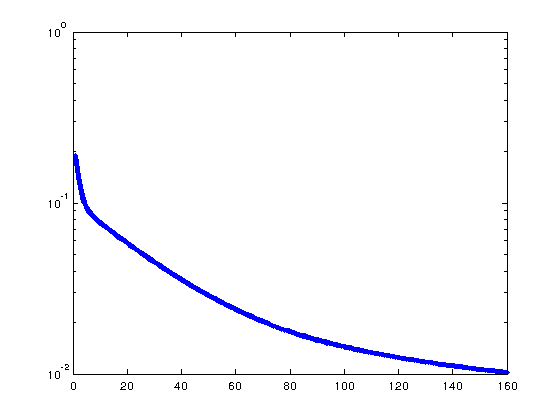}
\includegraphics[width=0.49\textwidth]{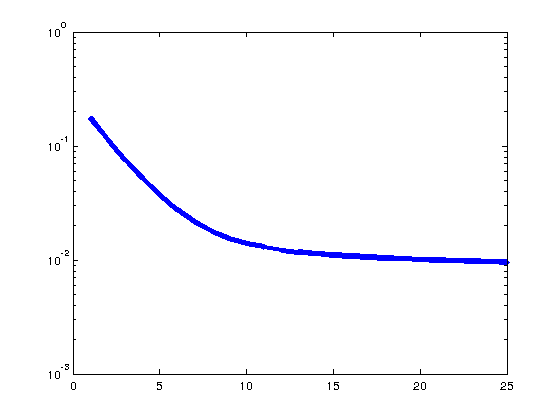}

\caption{Residual vs. number of iterations, until stopping criterion is reached, for $\alpha = 5\%$. Exact solution: left $u_1$, right $u_2$.}\label{fig_respol}

\end{figure}

\end{document}